\documentclass[preprint,none]{imsart}
\RequirePackage[OT1]{fontenc}
\RequirePackage{amsthm,amsmath,amssymb,natbib}
\usepackage{color%,showkeys
}

\newcommand{\mc}{\mathcal}
\newcommand{\mb}{\mathbb}
\newcommand{\E}{\mathbb{E}}
\newcommand{\R}{\mathbb{R}}

\newcommand{\ti}{\widetilde}

\newcommand{\I}{\textbf{1}_}
\newcommand{\eps}{\varepsilon}
\newcommand{\lo}{\langle}
\newcommand{\ro}{\rangle}
\newcommand{\U}{\mathcal{K}}
\newcommand{\mf}{\mathfrak}
\newcommand{\tr}{{\text{\tiny{\textsf{T}}}}}
\newcommand{\n}{{n}}
\theoremstyle{definition}
\newtheorem{defn}{Definition}[section]
\newtheorem{rmk}[defn]{Remark}

\newtheorem*{ak}{Acknowledgements}
\theoremstyle{plain}
\newtheorem{lem}[defn]{Lemma}
\newtheorem{thm}[defn]{Theorem}
\newtheorem{cor}[defn]{Corollary}

\newtheorem{ass}[defn]{Assumption}
\newtheorem{prop}[defn]{Proposition}
\numberwithin{equation}{section}
\DeclareMathOperator*\esssup{ess \,sup}
\DeclareMathOperator*\essinf{ess \,inf}
\DeclareMathOperator*\Limsup{Lim \,sup}
\DeclareMathOperator*\Liminf{Lim \,inf}
\DeclareMathOperator*\Lim{Lim}

\begin{document}
\begin{frontmatter}

% "Title of the paper"
\title{The Stability of the Constrained Utility Maximization Problem - A BSDE Approach}
\runtitle{Sensitivity under Cone Constraints}

\author{\fnms{Markus} \snm{Mocha\hspace{-0.15cm}}\ead[label=e2]{mocha@math.hu-berlin.de}}
\address{Markus Mocha\\ Institut f\"{u}r Mathematik\\
Humboldt-Universit\"{a}t zu Berlin \\ Unter den Linden 6, 10099 Berlin \\ Germany\\
\printead{e2}}
\and
\author{\fnms{Nicholas} \snm{Westray\thanksref{t4}}\ead[label=e4]{n.westray@imperial.ac.uk}}
\address{Nicholas Westray\\ Department of Mathematics\\ Imperial College\\  London SW7 2AZ \\ United Kingdom\\
\printead{e4}}\thankstext{t4}{Corresponding author: \texttt{n.westray@imperial.ac.uk}}

%\thankstext{t4}{This research was funded by an EPSRC DTA Grant}

\affiliation{Humboldt-Universit\"{a}t zu Berlin and Imperial College London}

\runauthor{M. Mocha and N. Westray}


\begin{abstract}
This article studies the sensitivity of the power utility maximization problem with respect to the investor's relative risk aversion, the statistical probability measure, the investment constraints and the market price of risk. We extend previous descriptions of the dual domain then exploit the link between the constrained utility maximization problem and continuous semimartingale quadratic BSDEs to reduce questions on sensitivity to 
results on stability for such equations. This then allows us to prove appropriate convergence of the primal and dual optimizers in the semimartingale topology.
\end{abstract}
\end{frontmatter}

%%%%%%%%%%%%%%%%%%%%%%%%%%%%%%%%%%%%%%%%%%%%%%%
%Introduction for utility valued on half line
%%%%%%%%%%%%%%%%%%%%%%%%%%%%%%%%%%%%%%%%%%%%%%%%%%

\section{Introduction}
In this article we study the optimal investment problem for an agent over a horizon interval $[0,T]$. 
The strategies or portfolios available to the agent 
are those which are valued in a convex cone, representing constraints like no short selling, and the aim is 
to maximize the expected utility of terminal wealth $\E[U(X_T)]$, where the utility function $U$ is of power type
and models the agent's preferences. The question of existence and uniqueness
for such a problem is classical in mathematical finance and has been extensively studied; in particular it is known that conditions exist
guaranteeing a unique solution. The focus in the present paper is on stability, namely we want to address the following question.

\begin{quote}
\normalsize{``Do the components of the solution, such as the optimal terminal wealth and investment strategy, depend continuously on the 
	input parameters, e.g.  utility function, asset price dynamics, investment constraints?"}    
\end{quote}     

This research is motivated by both practical applications as well as theory. Consider a situation where the optimal
investment portfolio above is implemented, typically there will be small errors in the calibration of input parameters. 
In order that the usefulness of performing such an optimization is not diminished it is necessary to show that such errors do not largely 
affect the optimizers, at least locally, which ties in with the above question.

There is a huge volume of literature related to utility maximization going back as far as Merton \cite{Me69,Me71}, for an excellent overview of the case where there are no investment constraints we refer to the survey article of Schachermayer \cite{S04} as well as the references therein. The situation where there are cone constraints has been studied more recently, in particular existence and uniqueness are still guaranteed. We refer the interested reader to the articles of Cuoco \cite{Cu97} and Cvitani{\'c} and Karatzas \cite{CK93} for It\^{o}-price dynamics 
and Karatzas and {\v{Z}}itkovi{\'c} \cite{KZ03}, Mnif and Pham \cite{MP01} and Westray \cite{We09} for the case of semimartingale dynamics. 
The modern solution approach for both constrained and unconstrained problems is via the duality or martingale method, where the convexity of the problem as well as the link between (a generalization of) martingale measures and replicable wealths is exploited. 

It is by the study of this dual problem that the mathematical literature on stability has proceeded thus far, beginning
with the article of Jouini and Napp \cite{JN04} and developing subsequently into two themes.
The first analyses continuity with respect to the preferences. A
sequence of utility functions (not necessarily of power type) $U^n$ converging to $U$ is considered and the continuity of the 
corresponding optimizers investigated, for complete It\^{o}-price models in \cite{JN04} and for incomplete markets with 
continuous semimartingale dynamics in Larsen \cite{La09}. In the complete case of \cite{JN04}, due to greater structure on the problem,
the authors prove the $L^p$ convergence of the optimal terminal wealth and strategy, whereas in \cite{La09} this is weakened to convergence in probability of only the optimal terminal wealth. More recently Kardaras and {\v{Z}}itkovi{\'c} \cite{KZ07} show that such convergence in probability of the optimal terminal wealth also holds when there is a random endowment and the statistical probability measure simultaneously varies, modelled by a sequence of measures $\mb{P}^n$ converging in total variation norm. Finally we mention the work by Nutz \cite{Nu10} who looks at risk aversion asymptotics for the power utility function, but also provides results on the continuity with respect to the risk aversion parameter.

The second theme, beginning with Larsen and {\v{Z}}itkovi{\'c} \cite{LZ07}, relates to misspecifications in the model, i.e. the utility function is fixed (again, not necessarily of power type) and the asset price dynamics vary. Typically there is a continuous semimartingale $S^\lambda$, 
modelling the financial market, which is indexed by its market price of risk $\lambda$. A sequence $\lambda^n$ is then chosen, appropriately convergent to some $\lambda$, and the convergence of the optimal terminal wealths $\hat{X}_T^{\lambda^n}$ is studied, again in probability. Continuity is shown under a suitable uniform integrability assumption and the results therein have recently been generalized to the conditional value functions and optimal wealth random variables $\hat{X}^{\lambda^n}_{\tau}$ for a stopping time $\tau$ valued in $[0,T]$, we refer to Bayraktar and Kravitz \cite{BK10} for further details.

The previous articles consider stability/continuity only in the situation where there are no investment constraints. 
In the specific case when the utility function is the logarithm, this can be generalized as shown in a recent article by Kardaras \cite{K10}. The optimizing investment strategy is then called the num{\'e}raire portfolio and by using its known explicit formula it is shown to depend continuously on the filtration, probability measure as well as the investment constraints, modelled by a sequence of cones. 

For the case of power, logarithmic and exponential utility functions recent literature, see the articles by Hu, Imkeller and M\"{u}ller \cite{HIM05}, Morlais \cite{Mo09} and Nutz \cite{Nu109,Nu209}, has focussed on an alternative approach to solving the utility maximization problem. In this case the value function admits a factorization property and it is possible to reduce the study of the optimal wealth process and investment strategy to the study of the solution of a quadratic semimartingale backward stochastic differential equation (BSDE), even in the presence of constraints. It is this correspondence which is exploited in the present article, showing that questions of sensitivity for the optimal wealth process and investment strategy are directly related to stability results for semimartingale BSDEs recently established in Mocha and Westray \cite{MW10}. One of the main features of the present article is that we work under an  exponential moments condition, rather than a boundedness condition, on the mean variance tradeoff process.

The focus in the present article is specifically on power utility and thus our results are simultaneously more and less general than previous literature. We are fixed within a class of utility functions but allow for continuous semimartingale dynamics and constraints. Via the link with BSDEs we can simultaneously consider continuity with respect to utility function, model dynamics, statistical probability measure and cone constraints, integrating previous research into one framework. Here, the convergence used is that of the semimartingale topology, hence directly on the level of processes as opposed to convergence of the terminal wealth random variable in probability, as is typically shown.

%\a{In this article we investigate the investor's portfolio choice problem under cone constraints which we assume to be \emph{stochastic}. We can think of the following example where our framework is suited for application. Consider a market in which a regulator bans the short selling of certain stocks as soon as some reference levels are attained. For instance, if stock prices drop dramatically or some volatility numbers indicate a sharp decline in the market participants' confidence the regulatory framework might provide for a ban of short selling to discourage aggressive speculators. Similarly, the board of an investment, pension or hedge fund may impose rules on the fund manager's investment possibilities that come into effect when the market exhibits a specified behaviour.}

The main contributions of this paper are divided into two parts. The first half provides a one to one relationship between the optimal wealth, strategy and dual variable and the solution to a quadratic semimartingale BSDE. This connection is proved in the presence of predictably measurable conic constraints on the investment strategy. The main tool used here is a result on the decomposition of elements of the dual domain, which extends those of \cite{KZ03} and \cite{LZ07} to allow for semimartingale dynamics and predictably measurable cone constraints. Such a result thus adds to the convex duality literature.   

The second half of this article applies this correspondence to study the continuity of the optimizers. The main contribution is to prove that the optimal wealth, strategy and dual variable all depend continuously on the input parameters of risk aversion, market price of risk, probability measure and constraints. We show that this convergence takes place in the semimartingale topology which extends the results in \cite{KZ07}, \cite{La09} and \cite{LZ07}. A feature of this result is that we rely on BSDE techniques rather than duality theory, which is new in the literature in this area. A final contribution is an example which shows that our conditions are necessary.  

%\a{In the motivating example above the regulator may well be concerned about the effect of having their decision taken under misspecifications of the market dynamics. Also, in reality, every investor can only have a guess about their own risk aversion parameter, and typically the stock dynamics and probability measures are based on statistical estimation. Both the regulator and the investor hence must rely on the fact that no degeneracies occur, i.e. on the stability of the portfolio choice problem, and we provide a mathematical framework together with a justification. More precisely, we derive convergence of the optimal strategies, as well as convergence of the optimal wealth and dual processes in the semimartingale topology, which is stronger than convergence at terminal time in probability. This improves the statements in \cite{KZ07}, \cite{La09} and \cite{LZ07}. Our final contribution is a counterexample showing that our conditions are tight.} 

The structure of the article is as follows, the modelling framework and main results are described in Sections \ref{chp_Uhalfline_model}
and \ref{sec_sensitivity}. Sections \ref{sec_dualdom} and \ref{sec_linkBSDEs} discuss the description of the dual domain and relationship between
the utility maximization problem and the solution to an appropriate BSDE. The connection with continuity is then shown in Section \ref{sec_stability}.
Related results whose proof would interrupt the flow of the text are collected
in the appendices.

%%%%%%%%%%%%%%%%%%%%%%%%%%%%%%%%%%%%%%%%%%%%%%%%%%%%
%Model Formulation
%%%%%%%%%%%%%%%%%%%%%%%%%%%%%%%%%%%%%%%%%%%%%%%%%%%%

\section{Model Formulation}
\label{chp_Uhalfline_model}
Throughout the present article we work on a
filtered probability space $(\Omega,\mc{F},(\mc{F}_t)_{0\leq t\leq
T},\mb{P})$ satisfying the usual conditions of right-continuity
and completeness. We assume that the time horizon $T$ is a finite number in $(0,\infty)$ and
that $\mc{F}_0$ is the completion of the trivial $\sigma$-algebra. All semimartingales are considered to be
equal to their c{\`a}dl{\`a}g modification. In order to apply the techniques of BSDE theory we will need the following assumption, 
often referred to in the literature as \emph{continuity of the filtration}. 
\begin{ass}\label{ass_contfilt}
All local martingales are continuous.
\end{ass}

There is a market consisting of one bond, assumed constant, and $d$ stocks with
discounted price process $S=(S^1,\ldots,S^d)^\tr\!$, a $d$-dimensional continuous
semimartingale on the given stochastic basis where we write ${}^\tr$ for transposition. More precisely, our
semimartingale $S$ is assumed to have dynamics
\begin{equation*}
dS_t=\text{Diag}(S_t)\,\big(dM_t+ \,d\lo M,M \ro_t\lambda_t\big),
\end{equation*}
where $M=(M^1,\ldots,M^d)^\tr$ is a $d$-dimensional continuous local martingale with $M_0=0$, $\lambda$ is a $d$-dimensional predictable process, the \emph{market price of risk}, satisfying
\begin{equation*}
%\mb{P}\left(\lo \lambda\cdot M,\lambda\cdot M\ro_T <+\infty\right)=1 
\mb{P}\left(\int_0^T \lambda_t^\tr \,d\lo M,M \ro_t\lambda_t <+\infty\right)=1
\end{equation*}
and $\text{Diag}(S)$ denotes the $d\times d$ diagonal matrix having elements taken from $S$. Observe that $\cdot$ denotes stochastic integration and we write $\lo M,M\ro$ for the quadratic \mbox{(co-)}variation matrix of $M$. Note that it is a consequence of Delbaen and Schachermayer \cite{DS95} Theorem 3.5 that any continuous, arbitrage free, num\'{e}raire denominated model of a market is of the above form so there is no loss of generality in the above framework.

To precisely describe our model we need some further results on $\lo M,M\ro$. We may use Jacod and Shiryaev \cite{JS03} Proposition II.2.9 and II.2.29 to write
\begin{equation}
\label{eq_decompQVM}
\lo M,M\ro=C\cdot A,
\end{equation}
where $C$ is a predictable process valued in the space of symmetric positive semidefinite $d\times d$ matrices and $A$ is a predictable increasing process. It is known that there are many such factorizations, cf. \cite{JS03} Section III.4a. We can choose $A:=\arctan\!\left(\sum_{i=1}^d\lo M^i,M^i\ro\right)$ and then, following an application of the Kunita-Watanabe inequality, we may derive the absolute continuity of each $\lo M^i,M^j\ro$ with respect to $A$ to get $C$. From Karatzas and Shreve \cite{KS91} Theorem  3.4.2 it is known that there exist Borel measurable functions which diagonalize a symmetric positive semidefinite $d\times d$ matrix, in particular we
deduce the existence of some processes $P$ and $\Gamma$ valued in the space of $d\times d$
orthogonal (resp. diagonal) matrices such that
\begin{equation}
\label{eq_decompQVM2}
\lo M,M \ro =C\cdot A=P^\tr \Gamma P\cdot A=B^\tr B\cdot A,
\end{equation}
where we set $B:=\Gamma^{\frac{1}{2}}P$. The matrix $\Gamma$ has nonnegative entries only, with the eigenvalues of $C$ on its diagonal. We also point out that our results do not depend on the particular choice of $A$, but only on its boundedness. The above processes $A,B,C,P$ and $\Gamma$ will be fixed throughout.

We use $\mc{P}$ to denote the predictable $\sigma$-algebra on $[0,T]\times \Omega$, generated by all the left-continuous adapted processes. It is known that the process $A$ induces a measure $\mu^A$ on $\mc{P}$, the \emph{Dol{\'e}ans measure}, defined for $E\in\mc{P}$ by
\begin{equation}
\mu^A(E):=\E\!\left[\int_0^T\I{E}(t)\,dA_t\right].
\end{equation}

We use the abbreviation $\Upsilon$ for a process $(\Upsilon_t)_{0\leq t\leq T}$ and write ``for
all $t$" implicitly meaning ``for all $t\in[0,T]$". A local martingale $N$ is called \emph{orthogonal} to $M$ if $\lo M^i,N^c\ro\equiv0$ for all $i=1,\ldots,d$ where $N^c$ denotes the \emph{continuous part} of $N$.
We refer the reader to \cite{JS03} and Protter \cite{Pr04} for any unexplained terminology and background material.

For the present article we require the following assumption.
\begin{ass}
\label{ass_expmom}
For all $c>0$ we have that 
%\begin{equation*}
$\E\big[\exp\!\left(c\,\lo \lambda\cdot M,\lambda\cdot M\ro_T\right)\big]%=\E\!\left[\exp\!\left(c\int_0^T \lambda^\tr_t \,d\lo M,M \ro_t\lambda_t\right)\right]
<+\infty.$
%\end{equation*}
\end{ass}
We describe this by saying that the mean-variance tradeoff $\lo \lambda\cdot M,\lambda\cdot M\ro_T$ has \emph{exponential moments of all orders}.

\begin{rmk}
Assumption \ref{ass_expmom} allows us to provide a unified presentation of the duality and the BSDE approach to solving the utility maximization problem; it ensures the existence of an equivalent local martingale measure for $S$ as well as implying finiteness of the primal and dual problems. Our analysis involves semimartingale BSDEs for which the above condition allows us to apply the existence, uniqueness and stability results from Appendix \ref{appendBSDE}. 
\end{rmk}

Trading in the above market is subject to constraints which we now describe. Recall that an $\mb{R}^d$-valued \emph{multivalued mapping} $G$ is a function $G:[0,T]\times\Omega\to2^{\mb{R}^d}$ (the power set of $\mb{R}^d$). It is called \emph{predictably measurable} if, for all closed subsets $Q$ of $\mb{R}^d$, 
\begin{equation*}
G^{-1}(Q):=\left\{(t,\omega)\in[0,T]\times\Omega~|~ G(t,\omega)\cap Q\neq\emptyset\right\}\in\mc{P}.
\end{equation*} 
The function $G$ is called closed (convex) if $G(t,\omega)$ is a closed (convex) set for
all $(t,\omega)\in [0,T]\times\Omega$. The constraints are modelled by
the predictably multivalued mapping $\mc{K}$ and we assume it satisfies the following assumption. 

\begin{ass}
\label{asscones}
The mapping $(t,\omega)\mapsto\U(t,\omega)\subset\R^d$ is closed, convex, and polyhedral in the following sense. There is an integer $m\geq1$, independent of $(t,\omega)$, together with corresponding predictable $M$-integrable processes $K^{1},\ldots,K^{m}$ such that $\mb{P}$-a.s. for all $t\in[0,T]$
\begin{equation*}
\U(t,\omega)=\left\{\sum_{j=1}^{m} c_{j}\, K^{j}_t(\omega)\,\Bigg|\,c_{j}\geq0,j=1,\ldots,m\right\}.
\end{equation*}
\end{ass}

Further discussion and explanation on the importance of the above assumption on $\U$ from the point of view of existence of optimal strategies can be found in Czichowsky and Schweizer \cite{CS09} as well as Czichowsky, Westray and Zheng \cite{CWZ09}. The unconstrained case is covered by setting $\U\equiv\mb{R}^d$. Other special cases include a constant polyhedral cone in $\mb{R}^d$ as in \cite{KZ03}, as well as $\U\equiv\mb{R}^{d_1}\times\{0\}^{d_2}$ with $d=d_1+d_2$ in which we face a model where the processes $S^{d_1+1},\ldots,S^{d}$ are \emph{nontradable}. 

We are now ready to introduce the notion of a trading strategy.
\begin{defn}
A predictable $d$-dimensional process $\nu$ is called \emph{admissible} and we write $\nu\in\mc{A}_\U$, if
\begin{enumerate}
\item It is $M$-integrable, i.e.
\begin{equation*}
%\mb{P}\left(\lo \nu\cdot M,\nu\cdot M\ro_T <+\infty\right)=1.
\mb{P}\left(\int_0^T\nu^\tr_t \,d\lo M,M \ro_t\nu_t <+\infty\right)=1.
\end{equation*}
\item We have that $\nu\in\U$, $\mu^A$-a.e.
\end{enumerate}
\end{defn}

In our framework, an admissible process $\nu$ will be interpreted as an investment strategy and its components $\nu^i$ represent the \emph{proportion of wealth} invested in each stock $S^i$, $i=1,\ldots,d$, subject to \emph{investment constraints} that are determined by $\U$. In particular, for some initial capital $x>0$ and an admissible strategy $\nu$, the associated  \emph{wealth process} $X^{x,\nu}$ evolves as follows
\begin{equation}
X^{x,\nu}:=x\,\mc{E}(\nu\cdot M +\nu\cdot \lo M,M\ro\lambda),\label{wealth}
\end{equation}
where $\mc{E}$ denotes the stochastic exponential. The family of all wealth processes arising from admissible strategies will be denoted by $\mc{X}(x)$, where we suppress the dependence on $\U$. Furthermore, we will omit writing explicitly the dependence of $X^{x,\nu}$ on the initial capital, when no ambiguity arises we just write $X^\nu$.
%We consider constrained utility maximization on $\mb{R}_+:=(0,\infty)$ so that there is a natural positivity condition; bearing this in mind we write things in exponential format. More specifically, an admissible $\nu$ will be 

\begin{rmk}\label{addmultformat}
The wealth equation is often written in additive format, $X=x+H\cdot S$ for a predictable $S$-integrable process $H$ specifying the amount of the asset held in the portfolio and chosen such that it is valued in some constraint set and the resulting wealth process remains (only) \emph{nonnegative}. We write $\mc{X}^{add}(x)$ for such wealth processes and observe that $\mc{X}(x)\subset\mc{X}^{add}(x)$. In the case that $X_T>0$, which implies $X>0$ since $X$ is a supermartingale under some equivalent measure (assumed to exist), the correspondence between $H$ and $\nu$ is given by $H^iS^i=\nu^i X$ for $i=1,\ldots,d$. The cone constraint in the additive formulation consists of the requirement that $H\in \mathcal{L}$ where 
\begin{equation*}
\mathcal{L}(t,\omega):=\left\{\sum_{j=1}^{m} c_{j}\, L^{j}_t(\omega)\,\Bigg|\,c_{j}\geq0\right\} 
\end{equation*}
with $\mathbb{R}^d$-valued predictable $S$-integrable processes $L^1,\ldots,L^m$ such that the $i$th component of each $K^j$ equals $S^i$ times the $i$th component of $L^j$. In particular the framework of \cite{KZ03}, where $\mc{L}$ is \emph{constant}, can be embedded into ours since we allow for a \emph{predictably measurable} multivalued mapping $\U$.
\end{rmk}
\begin{rmk}
Our motivation for writing wealth in exponential format stems from the fact that the dual domain of the portfolio choice problem will be a family of supermartingale measures, hence stochastic exponentials. It then turns out that to describe the primal and dual optimizers via a BSDE it is most convenient to write wealth also as a stochastic exponential. An additional byproduct of this parameterization is that it simplifies the proof of the decomposition of the elements of the dual domain.

Since in our setting the optimal wealth $\hat{X}$ exists and satisfies $\hat{X}_T>0$ we may, without loss of generality, choose to optimize over the family of (strictly) positive wealth processes $\mc{X}(x)$. This, together with the fact that $\mc{K}$ and $\mc{L}$ are predictably measurable multivalued mappings, allows one to switch freely between the two formulations of the wealth process. 
\end{rmk}

Our agent has preferences modelled by a utility function $U$, which is here assumed to be of power type, $U(x)=\tfrac{\,x^p}{p},$
for $p\in(-\infty,0)\cup(0,1)$. They start with initial capital $x>0$, may choose admissible strategies $\nu$,
and aim to maximize the expected utility of terminal wealth. This leads to the following formulation
of the primal optimization problem, 
\begin{equation}
\label{primalproblem} u(x):=\sup_{\nu\in\mc{A}_\U}\,\E\!\left[U\Big(X^{x,\nu}_T\Big)\right].
\end{equation}
\begin{rmk}
A key property arising under power utility and to be used throughout is the factorization property of the value function, more precisely we can write
\begin{align*}
u(x)&=x^p\,\sup_{\nu\in\mc{A}_\U}\E\!\left[U\Big(X^{1,\nu}_T\Big)\right] =U(x)\,c_p,
\end{align*}
for some constant $c_p$, $p\in(-\infty,0)\cup(0,1)$, to be identified below. A well known corollary of this is that the optimal investment strategy $\hat{\nu}$, when it exists, is independent of $x$ and the primal optimizer $\hat{X}$ has a
simple linear dependence on $x$.
\end{rmk}
Related to the above primal optimization problem is a dual problem which we now describe. For $y>0$ we introduce the set
\begin{equation*}
\mc{Y}(y):=\left\{Y\geq 0\,|\,Y_0=y \text{ and } XY \text{
is a supermartingale for all }X\in\mc{X}(1)\right\}
\end{equation*}
%Following Remark \ref{addmultformat} we define the ``additive'' dual set $\mc{Y}^{add}(y)\subset\mc{Y}(y)$ correspondingly (by using $\mc{X}^{add}(1)$ in the above equation) as well as 
and consider the minimization problem
\begin{equation}
\label{dualproblem}
\ti{u}(y):=\inf_{Y\in\mc{Y}(y)}\E\!\left[\ti{U}\big(Y_T\big)\right],
\end{equation}
where $\ti{U}$ is the \emph{conjugate} (or \emph{dual}) of $U$ given for $y>0$ by
\begin{equation*}
\ti{U}(y):=\sup_{x>0}\,\{U(x)-xy\}=-\tfrac{\,y^q}{q},
\end{equation*}
with $q:=\frac{p}{p-1}$ the dual exponent to $p$. Note that this set has the following factorization property, $\mc{Y}(y)=y\mc{Y}(1)$. Similarly to $u$ we then see the factorization property for $\ti{u}$,
\begin{align*}
\ti{u}(y)=\inf_{Y\in\mc{Y}(1)}\E\!\left[\ti{U}\big(yY_T\big)\right]
=y^q\,\inf_{Y\in\mc{Y}(1)}\E\!\left[\ti{U}\big(Y_T\big)\right]
=\ti{U}(y)\,\ti{c}_p.
\end{align*}
The relationship between $\ti{c}_p$ and $c_p$ is provided in Theorem \ref{thm_WZ09}. 

The utility maximization problem with general semimartingale dynamics and utility functions (not specifically power) has been studied under constant constraints, see \cite{KZ03} for the case with intertemporal consumption as well as \cite{We09}. The next proposition shows that the assumptions necessary to apply these results hold in our setting.
\begin{prop}
\label{prop_EMM}
Let Assumption \ref{ass_expmom} hold then there exists an equivalent local martingale measure for $S$ and $\max(u(x),\ti{u}(y))<+\infty$ for all $x,y>0$.
\end{prop}
\begin{proof}
The process $Y^\lambda:=\mc{E}(-\lambda\cdot M)$ is the density process of the so called \emph{minimal martingale measure}
thanks to our exponential moments condition and Novikov's criterion. 
For the second part we need only consider the case 
$p\in(0,1)$ and then from the definition of
$\ti{U}$ we have 
\begin{align*}
%u(x)&=\sup_{\substack{\nu\in\,\mc{A}_\U}}\E\!\Big[U\big(X_T^{x,\nu}\big)\Big]\leq
\max(u(x),\ti{u}(y))&\leq
\E\!\left[\ti{U}\big(yY_T^{\lambda}\big)\right]+
\sup_{\substack{\nu\in\,\mc{A}_\U}}\E\!\left[X_T^{x,\nu}y
Y_T^{\lambda}\right]
\leq -\tfrac{y^q}{q}\,\E\!\left[\big(Y_T^{\lambda}\big)^q\right]+
xy.
\end{align*}
The proof is completed by observing $q<0$ and using the H{\"o}lder inequality to derive
\begin{equation*}
\E\!\left[\big(Y_T^{\lambda}\big)^q\right]
=\E\!\left[\mc{E}(-2q\lambda \cdot
M)_T^{1/2}\exp\!\Big(q(2q-1)\lo\lambda\cdot M,\lambda\cdot M\ro_T\Big)^{1/2}\right]%\\&\leq\E\!\left[\exp\!\left(q(2q-1)\int_0^T\lambda_t^\tr\,d\lo
%M,M\ro_t\lambda_t\right)\right]^{1/2}
<+\infty.\qedhere
\end{equation*}
\end{proof}
The following theorem states the existence and uniqueness results that are pertinent for our study.
\begin{thm}
\label{thm_WZ09}
Suppose Assumptions \ref{ass_expmom} and \ref{asscones} hold and let $x,y>0$. Then:
\begin{enumerate}
\item There exists an admissible strategy $\hat{\nu}\in\mc{A}_\U$ which is optimal for the primal problem,
\begin{equation*}
u(x)=\E\big[U\big(\hat{X}_T\big)\big], \text{ where } \hat{X}=X^{x,\hat{\nu}}\!.
\end{equation*}
In addition, $\hat{\nu}$ is unique in the sense that for any
other optimal strategy $\bar{\nu}\in\,\mc{A}_\U$ %which is also
%optimal for the primal problem %satisfies
%\[\E\!\left[\int_0^T(\hat{\nu}_t-\bar{\nu}_t)^\tr\,d\lo
%M,M\ro_t(\hat{\nu}_t-\bar{\nu}_t)\right]=0,\]
%so that
the wealth processes $X^{x,\hat{\nu}}$ and $X^{x,\bar{\nu}}$ are indistinguishable. 
\item There exists an optimal $\hat{Y}^y\in\mc{Y}(y)$ for the dual problem, unique up to indistinguishability,
\begin{equation*}
\ti{u}(y)=\E\!\left[\ti{U}\big(\hat{Y}_T\big)\right]\!, \text{ where } \hat{Y}=\hat{Y}^{y}. \end{equation*}
\item The functions $u$ and $\ti{u}$ are continuously differentiable and conjugate. If $y=u'(x)$ then, adopting the notation from (i) and (ii), we have the relation
%\begin{equation*}
%\E[\hat{X}_T\hat{Y}_T]=xy,
$\hat{Y}_T = U'(\hat{X}_T)$
%\quad u(x)=\ti{u}(y)+xy
%\end{equation*}
and $\hat{X}\hat{Y}$ is a martingale on $[0,T]$. More explicitly, there are constants $c_p$, $p\in(-\infty,0)\cup(0,1)$, such that with $\ti{c}_p:=(c_p)^{\frac{1}{1-p}},$
\begin{equation*}
u(x)= U(x)\,c_p,\quad
\ti{u}(y)=\ti{U}(y)\,\ti{c}_p.
\end{equation*}
%\item If $y=u'(x)$ then the process $\hat{X}\hat{Y}$ is a martingale on $[0,T]$.
\end{enumerate}
\end{thm}
%\begin{proof}
%For the remaining items we refer to Appendix \ref{existappend}.
%\end{proof}
% Suppose we have a sequence% $(p_n,\lambda^n)_{n\in\mb{N}}$ together with a pair $(p,\lambda)$, $p\neq0$ such that
% \begin{gather*}
% \lim_{n\to\infty}|p_n-p|=0, \\
% \int_0^T(\lambda_t^n-\lambda_t)^\tr\,d\lo M,M \ro_t(\lambda_t^n-\lambda_t)
% \xrightarrow{\mb{P}}0.
%\end{gather*}
% The question addressed in the remainder of the paper is what can be said about the continuity of the mappings \begin{align*}
% (x,\lambda,p)\mapsto \hat{X}^{x,\lambda,p}\\
% (x,\lambda,p)\mapsto \hat{Y}^{((u)'(x),\lambda,p})\\
% (x,\lambda,p)\mapsto\hat{\nu}(x,\lambda,p)
% \end{align*}
% and in which topology. Before discussing continuity we must have a topology
% in mind. Since for $\hat{X}$ and $\hat{Y}$ we are considering , for the optimal primal and dual random  r $X$ and $Y$) we look at convergence on the level of processes so that the appropriate space is $\mb{D}$, c{\`a}dl{\`a}g processes equipped with the ucp topology in contrast to \cite{LZ07} who are interested in only the terminal value of the processes.

In the present article our aim is to analyze the above problems and their stability by relating them directly to the solution of a continuous semimartingale BSDE of the following type:
\begin{gather}
\label{BSDE}
d\Psi_t=Z_t^\tr\,dM_t+dN_t-F(t,Z_t)\,dA_t-\frac{1}{2}\,d\lo N,N\ro_t, \quad \Psi_T=0,
\end{gather}
where $F$ is a predictable function $[0,T]\times\Omega\times\mb{R}^d\to\mb{R}$ called the \emph{generator} or \emph{driver}. 
\begin{defn}\label{DefnBSDESol}
A \emph{solution} to the BSDE \eqref{BSDE} is a
triple $(\Psi,Z,N)$ of processes valued in $\R\times\R^d\times\R$
satisfying \eqref{BSDE} $\mb{P}$-a.s. such that:
\begin{enumerate}
\item The function $t\mapsto \Psi_t$ is continuous $\mb{P}$-a.s.
\item The process $Z$ is predictable and satisfies $\int_0^TZ_s^\tr\,d\lo M,M\ro_sZ_s<+\infty$, $\mb{P}$-a.s.
\item The local martingale $N$ is continuous and orthogonal to $M$.
\item We have $\mb{P}$-a.s. that $\int_0^T|F(t,Z_t)|\,dA_t+\lo N,N\ro_T<+\infty.$ 
\end{enumerate}
We call $Z\cdot M+N$ the \emph{martingale part} of a solution.%In particular an implicit requirement on a solution to \eqref{BSDE} is that 
%\begin{equation*}
%\int_0^T|F(t,Z_t)|\,dA_t+\lo N,N\ro_T<+\infty.
%\end{equation*}
\end{defn}
We shall be especially interested in solution triples $(\Psi,Z,N)$ with 
$\Psi\in\mathfrak{E}$, where $\mf{E}$ denotes the space of processes $\Upsilon$
such that
\begin{equation*}
\E\!\left[\exp\!\left(c\Upsilon^*\right)\right]<+\infty \text{
for all } c>0,
\end{equation*}
i.e. those processes whose supremum, $\Upsilon^*:=\sup_{0\leq t \leq T} |\Upsilon_t|$, possesses exponential moments of all orders. Indeed, we are going to show that $\hat{\Psi}:=\log\bigl(\hat{Y}/U'(\hat{X})\bigr)$ is the unique solution to a specific quadratic BSDE with $\hat{\Psi}\in\mathfrak{E}$.

%\a{For the convenience of the reader we summarize the previous discussion and the reasoning of the present article in the following table.} 
%\begin{table}[H]
%\centering
%\begin{tabular}{|p{7.4cm}|p{7.4cm}|}\hline
%\begin{center}\textbf{Duality Approach}\end{center} & \begin{center}\textbf{BSDE Approach}\end{center} \\ 
%\multicolumn{2}{|c|}{Under the \textbf{exponential moments} assumption \ref{ass_expmom},}\\ \multicolumn{2}{|c|}{\phantom{Under the exponential moments assumption \ref{ass_expmom},}}\\
%$\bullet$ There \textbf{exist} primal and dual optimizers $\hat{X}$ and $\hat{Y}$. &$\bullet\!$ \textit{Quadratic BSDEs allow for solutions} $(\Psi,Z,N)$ \phantom{$\bullet$ }\;\textit{with }$\Psi\in\mathfrak{E}$.\\
%$\bullet$ The process $\hat{\Psi}:=\log(\hat{Y}/U'(\hat{X}))$ is part of a \textbf{solu-}\phantom{$\bullet$ }\textbf{tion} $(\hat{\Psi},\hat{Z},\hat{N})$ of a quadratic BSDE. &{}\\$\bullet$ We have that $\hat{\Psi}\in\mathfrak{E}$.&$\bullet$ Solutions with $\Psi\in\mathfrak{E}$ are \textbf{unique}.\\
%{}&$\bullet$ Solutions with $\Psi\in\mathfrak{E}$ hence \textbf{coincide} with those \phantom{$\bullet$ }from the duality approach.\\
%&$\bullet$ Quadratic BSDEs allow for \textbf{stability} results.\\\hline
%\end{tabular}
%\vspace{2mm}
%\caption{Link between the Duality and the BSDE Approach to the Study of the Existence, Uniqueness \\and Stability of the Utility Maximization Problem under Exponential Moments}
%\label{DualBSDEComparison}
%\end{table}
We conclude this section with the some notation and useful results.
For $\rho\geq1$ we write $\mc{M}^\rho$ for the space of $\mb{P}$-local martingales, where $\ti{M}\in\mc{M}^\rho$ if it is a local martingale satisfying $\ti{M}_0=0$ and
\begin{equation*}
\E\!\left[\lo \ti{M},\ti{M}\ro_T^{\rho/2}\right]<+\infty.
\end{equation*}
More generally for an arbitrary continuous semimartingale $\Upsilon$ we shall indirectly use the $\mc{H}^\rho$ norm. Given the 
canonical decomposition $\Upsilon=\Upsilon_0+M^\Upsilon+A^\Upsilon$ where $M^\Upsilon$ is a (continuous) local martingale and $A^\Upsilon$ a (continuous) process
of finite variation, it is defined via
\begin{equation*}
\|\Upsilon\|_{\mc{H}^\rho}:= | \Upsilon_0| + \Big\|\lo M^\Upsilon,M^\Upsilon\ro_T^{1/2}\Big\|_{L^\rho(\mb{P})} +
 \left\|\int_0^T\big|dA_s^{\Upsilon}\big|\right\|_{L^\rho(\mb{P})}\!.
\end{equation*}

The stability result that we are going to derive involves the notion of convergence in the semimartingale topology for which we refer the reader to {\'E}mery \cite{Em79} and M{\'e}min \cite{Me80} for more details. The following proposition collects together the key results needed in the present article.
\begin{prop}[{\'E}mery \cite{Em79} Lemma 6, Nutz \cite{Nu10} Appendix A]\label{semimarttop}
Let $(\Upsilon^n)_{n\in\mb{N}_0}$ be a family of continuous semimartingales and $\rho\geq1$, then 
\begin{enumerate}
\item The sequence $(\Upsilon^n)_{n\in\mb{N}}$ converges to $\Upsilon^0$ in the semimartingale topology if and only if every subsequence of $(\Upsilon^n)_{n\in\mb{N}}$ has a subsequence %still denoted by $(\Upsilon^n)_{n\in\mb{N}}$, 
converging locally to $\Upsilon$ in $\mc{H}^\rho$.
\item If $(\Upsilon^n)_{n\in\mb{N}}$ converges to $\Upsilon^0$ in the semimartingale topology then $(\mc{E}\big(\Upsilon^n)\big)_{n\in\mb{N}}$ converges to $\mc{E}\big(\Upsilon^0\big)$ in the semimartingale topology.
\item Convergence in the semimartingale topology implies convergence uniformly on compacts in probability, ucp in short, see \cite{Pr04} Section II.4.
\end{enumerate}
\end{prop}

One final notation we shall need is that of the \emph{polar cone}.
Given the conic predictably measurable multivalued mapping $\mc{K}$ we define (and easily derive)
\begin{align*}
\mc{K}^{\circ}(t,\omega):
=\left\{l\in\mb{R}^d \,\Big|\, k^{\tr}l\leq 1 \text{ for all } k\in\mc{K}(t,\omega)\right\}
= \bigcap_{j=1}^m\left\{l\in\mb{R}^d \,\Big|\, (K^j_t(\omega))^\tr l\leq 0  \right\},
\end{align*}
where the $K^j, j=1,\ldots,m$ are from Assumption \ref{asscones}. Under the present assumptions on $\mc{K}$ we have that $\mc{K}^\circ$ is again a closed convex predictably measurable multivalued mapping. 
%The above definition and second equality apply to any given cone, the third characterization holds only for polyhedral cones.

%%%%%%%%%%%%%%%%%%%%%%%%%%%%%%%%%%%%%%%%%%%%%%%%%%%%%%%%%%%%%%%%%%%%%%%%%%%%%%%%
%%%%%%%%%% Main results
%%%%%%%%%%%%%%%%%%%%%%%%%%%%%%%%%%%%%%%%%%%%%%%%%%%%%%%%%%%%%%%%%%%%%%%%%%%%%%%%
 
\section{Main Results}
\label{sec_sensitivity} 
Having described our framework and relevant background, we can state the first of the main results. Specifically, we provide a more precise structure of the elements in the dual domain as well as the optimizer $\hat{Y}$. A version of this result may be found in \cite{KZ03} Proposition 4.1 for the case where one has nondegenerate It\^{o} dynamics for $S$ and a polyhedral cone $\mc{K}$ which is independent of $(t,\omega)$, see also \cite{LZ07} Proposition 3.2 for the one-dimensional unconstrained case. We extend these results to the case of semimartingale dynamics and predictably measurable constraint sets. 
%\a{Due to writing wealth in exponential format the proof becomes simpler and more general when compared to \cite{KZ03} and \cite{LZ07}.}
\begin{thm} Let Assumption \ref{asscones} hold.
\label{thm_exist}
\begin{enumerate}
\item Let $Y\in\mc{Y}(1)$ with $Y_T>0$. Then there exist
an $M$-integrable process $\kappa^Y$ with 
\begin{equation*}
B(\lambda-\kappa^{Y})\in (B\mc{K})^\circ,\quad \mu^A \text{-a.e.}
\end{equation*}
as well as a local martingale $N^Y$ orthogonal to $M$ and a predictable
decreasing c{\`a}dl{\`a}g process $D^Y$ with $D^Y_0=1$ and $D^Y_T>0$ $\mb{P}$-a.s.
such that 
\begin{equation*}
Y=D^Y\,\mc{E}(-\kappa^Y\cdot M+N^Y).
\end{equation*}
\item For the optimizer $\hat{Y}^{y}$ (assumed to exist) we have the representation, 
\begin{equation*}
\hat{Y}^{y}=y\,\mc{E}(-\hat{\kappa}\cdot M +\hat{N}), 
\end{equation*}
for processes $\hat{\kappa}:=\kappa^{\hat{Y}}$ and $\hat{N}:= N^{\hat{Y}}$  which are independent of $y$. In particular the decreasing process from
 (i) satisfies $D^{\hat{Y}}\equiv 1$.
\end{enumerate}
\end{thm}
The next proposition relates the optimizers to the solution of a quadratic semimartingale BSDE, similarly to 
Mania and Tevzadze \cite{MT08} for the unconstrained case and \cite{Nu209} for the constrained case, see also \cite{HIM05} and \cite{Mo09}. 
\begin{prop}
\label{thm_bijection}
Let Assumptions \ref{ass_contfilt}, \ref{ass_expmom} and \ref{asscones} hold.
\begin{enumerate}
\item Let $\hat{\nu}$ denote the optimal strategy and $\hat{N}$ the local martingale from Theorem \ref{thm_exist}(ii). % is continuous.
Then for every $x>0$ the triple $(\hat{\Psi},\hat{Z},\hat{N})$, where
\begin{equation*}
\hat{\Psi}:=\log\!\left(\frac{u'(x)\hat{Y}^{1}}{U'(\hat{X})}\right)\quad\text{ and }\quad \hat{Z}:=-\hat{\kappa}+(1-p)\hat{\nu},
\end{equation*}
is the unique solution to the BSDE
\eqref{BSDE} with $\hat{\Psi}\in\mf{E}$ where
\begin{align*}
F(\cdot,z)=\frac{1}{2}\,\big\|Bz\big\|^2-\frac{q}{2}\,\big\|\Pi_{B\U}\big(B(z+\lambda)\big)\big\|^2.
\end{align*}
We write $\Pi$ for the nearest point or projection operator onto the indicated cone.
\item Given the unique solution $(\hat{\Psi},\hat{Z},\hat{N})$ from (i) we can write the optimizers, given initial values $x=1$, $y=1$, up to indistinguishability as  
\begin{align*}
\hat{X}^1=\mc{E}(\ti{\nu}\cdot M +\ti{\nu}\cdot \lo M,M\ro\lambda),\quad\quad
\hat{Y}^{1}=\mc{E}\big(-\ti{\kappa}\cdot M+\hat{N}\big),
\end{align*}
where the predictable integrands $\ti{\nu}$ and $\ti{\kappa}$
are defined via
\begin{align*}
\ti{\nu}:=\frac{1}{1-p}P^\tr\ti{\Gamma}^{\frac12}\left[\Pi_{B\mc{K}}\left(B(\hat{Z}+
\lambda)\right)\right],\quad
\ti{\kappa}:=
P^\tr\ti{\Gamma}^{\frac12}\left[B\lambda-\Pi_{(B\mc{K})^\circ}\left(B(\hat{Z}+\lambda)\right)\right]
\end{align*}
and satisfy, $\mu^A$-a.e. $B\ti{\nu}=B\hat{\nu}\text{ and }B\ti{\kappa}=B\hat{\kappa}. $
The process $(\ti{\Gamma}^{i,j})_{i,j=1,\ldots,d}$ is chosen to be a predictable process valued in the space of $d\times d$ diagonal matrices such that it satisfies
\begin{equation*}
\ti{\Gamma}^{ij}=
\begin{cases}
%\frac{1}{\Gamma^{ii}}
1\big/\Gamma^{ii}& \text{ if } i=j \text{ and }\Gamma^{ii}\neq0\\
0 & \text{ if }i\neq j.
\end{cases}
\end{equation*}
\end{enumerate}
\end{prop}
\begin{rmk} The content of the above proposition is essentially known, cf. Nutz \cite{Nu209} Corollaries 3.12 and 5.18, although it is stated differently there. Define the process $L^{\textit{op}}:=\exp(\hat{\Psi})$, then it is easy to see that $L^{\textit{op}}$ is the \emph{opportunity process} of \cite{Nu209} where the author shows that the Galtchouk-Kunita-Watanabe decomposition of $L^{\textit{op}}$ satisfies an appropriate BSDE. In contrast via our Theorem \ref{thm_exist}, we can apply It\^{o}'s formula directly to $\hat{\Psi}$ to derive the BSDE. As a consequence we augment the results of \cite{Nu209} by providing a simple additive decomposition of the process $\hat{Z}$ into a part $\hat{\kappa}$ with well defined properties related to the dual problem and polar cone $(B\mc{K})^{\circ}$ and a part related to the optimal strategy $\hat{\nu}$. We also note that the unique correspondence above is derived via a BSDE comparison theorem under an exponential moments condition, rather than a verification theorem as in \cite{Nu209}. We mention finally that the formula for the optimal strategy is as in the cited references, however, due to Moreau's decomposition theorem, which is available for cones, no measurable selection argument is involved. 
\end{rmk}

\begin{rmk}\label{rmk_nullinv}
We point out a consequence of item (ii) above. Whilst the wealth process is unique in the space of c\`{a}dl\`{a}g processes, the representation of the strategies is not unless $C$ is invertible or the strategy is considered in the image of $B$; similar remarks apply to $\hat{\kappa}$. This is related to the discussion of what is known as \emph{null-investments} in the literature.
% \a{and the above proposition identifies the components of the null-investments}. 
\end{rmk}

The main idea of the present article is to use the link from Proposition \ref{thm_bijection} to study the continuity of the utility maximization problem with respect to its inputs via BSDE methods. More explicitly, we are interested in the dependence of the optimal objects with respect to the market price of risk process $\lambda$, the probability measure $\mathbb{P}$, the investor's relative risk aversion parameter $p$ and the constraint set $\U$. 
As pointed out above, the dependence on the initial wealth is a simple linear one, due to the factorization property.
Hence we vary only the four inputs $\lambda$, $\mathbb{P}$, $p$ and $\U$ by means of sequences 
\[(\lambda^{n})_{n\in\mathbb{N}},\; (\mathbb{P}^{n})_{n\in\mathbb{N}},\; (p^{n})_{n\in\mathbb{N}} \text{ and } (\U^{n})_{n\in\mathbb{N}}\] 
of parameters that converge to $\lambda=:\lambda^0$, $\mathbb{P}=:\mathbb{P}^0$, $p=:p^0$ and $\U=:\U^0$ in an appropriate sense. 

Fix $n\in\mb{N}$, now we have that $\lambda^n$ is a predictable $M$-integrable process and $\mb{P}^{n}$ is assumed to be a measure equivalent to $\mb{P}$ with Radon-Nikod{\'y}m derivative
\begin{equation*}
\frac{d\mb{P}^{n}}{d\mb{P}}=\mc{E}(-\beta^{n}\cdot M+L^{n})_T.   
\end{equation*}
Here $(\beta^{n})_{n\in\mb{N}}$ is a sequence of $M$-integrable processes, $(L^{n})_{n\in\mb{N}}$ a sequence of continuous $\mb{P}$-local martingales orthogonal to $M$ and $\beta^{0}\cdot M:\equiv L^{0}:\equiv0$. Due to the Girsanov theorem the process
$M^{n}:=M+\lo M,M\ro\beta^{n}$ 
is a (continuous) $\mb{P}^{n}$-local martingale. This leads to dynamics for the asset $S=S^{n}$ under $\mb{P}^{n}$, of the form 
\begin{equation*}
dS_t^{n}=\mathrm{Diag}(S_t^{n})\Big(dM^{n}_t+d\lo M^{n},M^{n}\ro_t(\lambda^{n}_t-\beta^{n}_t)\Big),
\end{equation*}
where we have used the continuity to deduce $\lo M^{n},M^{n}\ro=\lo M,M\ro=C\cdot A$.
Each risk aversion parameter $p^{n}$ is valued in $(-\infty,0)\cup(0,1)$ and corresponds to a utility function 
\begin{equation*}
U^{n}(x):=\frac{1}{p^{n}}\,x^{p^{n}},\quad x>0.
\end{equation*}
The cone $\U^{n}$ is assumed to satisfy Assumption \ref{asscones} so that we can  
consider the primal problem as a function of the inputs 
\begin{equation*}
u^{n}(x):=\sup_{\nu\in\mc{A}_{\U^{n}}}\,\E_{\mb{P}^{n}}\!\!\left[U^{n}\Big(X^{{n},x,\nu}_T\Big)\right],
\end{equation*}
where $X^{{n},x,\nu}$ represents the wealth acquired from an investment in $S^{n}$ and considered under $\mathbb{P}^{n}$, so that we have
\begin{equation}
\label{eq_wealthUnderPn}
X^{{n},x,\nu}=x\,\mc{E}(\nu\cdot M +\nu\cdot \lo M,M\ro\lambda^{n}\big)=x\,\mc{E}(\nu\cdot M^{n} +\nu\cdot \lo M^{n},M^{n}\ro(\lambda^{n}-\beta^{n})\big).
\end{equation}
The definition of $\mc{A}_{\mc{K}^{n}}$ is invariant under changes of equivalent probability measures so that the above maximization is well defined under suitable assumptions on the parameters.
\begin{ass} 
\label{ass_unif.expmom}
Each $\U^{n}, n\in\mb{N}_0$, satisfies Assumption \ref{asscones} and for all $c>0$
\begin{equation*} 
\sup_{n\in\mathbb{N}_0}\E_\mathbb{P}\!\left[\exp\!\bigg(c\,\Big(\lo \lambda^n\cdot M,\lambda^n\cdot M\ro_T+\lo \beta^n\cdot M,\beta^n\cdot M\ro_T+\lo L^{n},L^{n}\ro_T\Big)\bigg)\right]<+\infty.
\end{equation*}
\end{ass}
The previous assumption ensures that for fixed $n\in\mathbb{N}_0$ and all $c>0$ \[\E_{\mathbb{P}^{n}}\!\!\left[\exp\!\left(c\,\big\lo (\lambda^{n}-\beta^{n})\cdot M,(\lambda^{n}-\beta^{n})\cdot M\big\ro_T\right)\right]<+\infty, \]
which can be shown similarly to the proof of Proposition \ref{prop_EMM} and is left to the reader. In particular we may apply Theorem \ref{thm_WZ09} for each ${n}\in\mb{N}_0$ to deduce the existence of a primal optimizer and corresponding
optimal portfolio,
\begin{equation}
\label{eq_optimizers}
\hat{X}^{n}:=\hat{X}(\lambda^{n},\mb{P}^{n},p^{n},\U^{n}),\quad \hat{\nu}^{n}:=\hat{\nu}(\lambda^{n},\mb{P}^{n},p^{n},\U^{n}), 
\end{equation}
where we write the optimizers as a function of the parameters $(\lambda^{n},\mb{P}^{n},p^{n},\U^{n})$. A similar convention holds for the value function $u^{n}:=u(\lambda^{n},\mb{P}^{n},p^{n},\U^{n})$. We also use the notation $\hat{X}:=\hat{X}^{0}$ and $\hat{\nu}:=\hat{\nu}^{0}$ and observe that due to the integrability assumption above $L^{n}$ is actually a true martingale for every $n\in\mb{N}_0$.

The third main result shows that under suitable assumptions the optimizers are continuous
with respect to these inputs. Note that in the following assumption since each $\U^n$ is polyhedral
the projection $B\U^n$ is closed.
\begin{ass}
\label{ass_prefconv}
The preferences and markets converge in the following sense
\begin{gather*}
\lim_{n\to+\infty}p^{n}=p,\\
\lim_{n\to+\infty}\big\lo (\lambda^{n}-\lambda)\cdot M, (\lambda^{n}-\lambda)\cdot M\big\ro_T 
+ \big\lo \beta^{n}\cdot M, \beta^{n}\cdot M\big\ro_T +
\big\lo L^{n}, L^{n}\big\ro_T=0
\end{gather*}
in $\mb{P}$-probability.
We assume that 
\begin{equation*}
\Lim_{n\to+\infty} B\U^{n}= B\U \text{ $\mu^A$-a.e.}, 
\end{equation*}
where $\Lim$ denotes the closed limit and we refer to Appendix \ref{appendKur} for more details.%, if this does not hold we set $\U^{n}\equiv\U$ for all $n\in\mb{N}$.  
\end{ass}
We can now state the main result of the present paper. 
\begin{thm}
\label{thm_mainstab}
Let Assumptions \ref{ass_contfilt}, \ref{ass_unif.expmom} and \ref{ass_prefconv} hold and $\hat{X}^{n}$ and $\hat{\nu}^{n}$ be as in \eqref{eq_optimizers}, then as ${n}\to +\infty$:
\begin{enumerate}
\item The sequence of processes $(\hat{\nu}^{{n}}-\hat{\nu})\cdot M$ converges to zero in $\mc{M}^2$.  
%on $(\Omega,\mc{F},\mb{P})$ as ${n}\to +\infty$.
\item The family of wealth processes $\hat{X}^n\in\mc{X}(x)$, $n\in\mb{N}$, converges to $\hat{X}\in\mc{X}(x)$ in the semimartingale topology.
\item The functions $u^{n}$ and their derivatives $(u^{n})'$ converge pointwise to $u$ and $u'$ respectively.
\end{enumerate}
\end{thm}

\begin{rmk}
We can establish \emph{identical} results for the corresponding sequence of dual problems and their optimizers.
However since these are not the main objects of interest we pursue them further in Section \ref{sec_stability}. % as
%a main result and instead 
%where the proofs of all the statements are given. We also note that from the above assumptions the sequence $(\mb{P}^n)_{n\in\mb{N}}$ converges to $\mb{P}$ in total variation.
\end{rmk}

\begin{rmk}
We discuss here in more detail how the theorem above relates to others in the literature. When $\beta^{n}\cdot M\equiv L^{n}\equiv0$, $p^{n}= p$ and $\mc{K}^{n}\equiv\mathbb{R}^d$ we are in the setting of \cite{LZ07}. Observe that our Assumption \ref{ass_unif.expmom} is more restrictive than the notion of ``V-relative compactness" introduced therein. Thus by fixing the utility and imposing stricter conditions on the $\lambda^{n}$ we get convergence of the whole path process together with the convergence of the optimal strategies, strengthening the main results of \cite{La09} as well as \cite{LZ07} where one gets convergence in probability of the optimal \emph{terminal} values $\hat{X}^{n}_T$. From the convergence in the semimartingale topology we deduce then the corresponding continuity results given in \cite{BK10} where $T$ is replaced by a stopping time $\tau$. 

When $\beta^{n}\cdot M\equiv L^{n}\equiv0$, $\lambda^{n}\equiv \lambda$ and $\mc{K}^{n}\equiv\mathbb{R}^d$
we recover \cite{Nu10} Corollary 5.7. Therein the process $S$ and filtration do not need to be continuous,
so that this result is more general than those presented here. The reason for this is that when only the risk aversion parameter varies one can compare the opportunity processes directly via Jensen's inequality. 
When $\mb{P},\lambda$ and $\mc{K}$ also vary such an approach seems not to be feasible, hence our reliance on BSDE methods alone which necessitates more stringent assumptions.   

For $\lambda^{n}\equiv \lambda$ and $\mc{K}^{n}\equiv\mathbb{R}^d$ observe that under our assumptions
$(\mb{P}^n)_{n\in\mb{N}}$ converges to $\mb{P}$ in total variation. Thus we recover \cite{KZ07} Theorem 1.5 
in the case where there is no random endowment and the utility is power. Similarly to the case of \cite{LZ07} above, our Assumption \ref{ass_unif.expmom} is more restrictive than Assumption (UI) therein. As a consequence we partially extend their results to convergence of the optimal wealth process in the semimartingale topology in a setting without random endowment.  

When $p=0$, $\lambda^{n}\equiv\lambda$ and, in addition to the cones and measure, the information structure is also allowed to vary, \cite{K10} obtains results similar to ours for the num\'{e}raire portfolio using its explicit formula. The problem there differs from ours as it is ``myopic" and as such there is no opportunity process and corresponding BSDE in the sense discussed here, so that one cannot directly compare the two approaches. We note however that in both cases the limiting cone is taken to be the closed limit of the sequence of cones $(\U^{n})_{n\in\mb{N}}$.

An approach similar to ours was used in \cite{Fr09} for the exponential indifference value when the mean-variance tradeoff process is bounded. There one has $\lambda^n\equiv\lambda$ and $\beta^{n}\cdot M\equiv L^{n}\equiv0$ so that the quadratic growth and locally Lipschitz assumptions on the respective BSDEs are uniform in $n$ so that a corresponding stability result can be used. 
%in the setting of a bounded mean-variance tradeoff that is present there. Moreover, the conditions required for stability to hold are on the drivers on the BSDEs and not in terms of the input parameters.}

As a final remark, when the utility function is allowed to vary, one typically needs to assume that the sequence converges pointwise and satisfies a uniform growth condition, see \cite{JN04,KZ07,La09}. This is implied by our Assumption \ref{ass_prefconv} so  that we are consistent with the literature in this respect.
\end{rmk}

\begin{rmk}
\label{rmk_KK}
Here we elaborate further on the type of convergence assumed on the cones. Together with Proposition \ref{kuraconv}, Assumption \ref{ass_prefconv}
implies that the projections $\Pi_{B\U^n}$ converge pointwise to $\Pi_{B\U}$ which is the key property in showing the convergence of the drivers of the related BSDEs. Define the set 
\begin{equation*}
\mathfrak{N}(t,\omega):=\ker\!\big(C(t,\omega)\big)=\ker\!\big(B(t,\omega)\big),
\end{equation*}
a closed predictably measurable multivalued mapping. This is the set of null-investments described in Karatzas and Kardaras \cite{KK07}.
In \cite{K10} the author replaces Assumption \ref{ass_prefconv} with 
\begin{equation*}
\mathfrak{N}\subset \U^n \text{ for all } n\in\mb{N}_0 \text{ and } \Lim_{n\to+\infty} \U^{n}= \U \quad \mu^A\text{-a.e.}
\end{equation*}
Proposition \ref{prop_Kard} shows that this is sufficient to imply $\Lim_{n\to+\infty} B\U^{n}= B\U$ $\mu^A$-a.e. %$\Pi_{B\U^n}$ converges to $\Pi_{B\U}$ pointwise 
so that the results of the present article remain valid under this alternative assumption. The requirement $\mathfrak{N}\subset \U^n \text{ for all } n\in\mb{N}_0$ means that although the investor faces investment constraints imposed on their portfolio. These constraints must be compatible with the null-investments in the sense that simultaneously the agent must be allowed to choose null-investment strategies. When $\ker(C)$ has a complicated structure this can be difficult to check and thus we prefer Assumption \ref{ass_prefconv}. Note that $\Lim_{n\to\infty} \U^{n}= \U$ alone is \emph{not} sufficient for the stability result to hold as is illustrated by a simple counterexample.
%in which the investment constraints are not compatible with a redundant structure of the market. Namely, 

Consider a standard one-dimensional Brownian motion $W$ and set $M:=(W,0)^\tr$. Taking a constant $\lambda=(\lambda_1,0)^\tr\in\mb{R}^2$, $\lambda_1>0$, completes the description of the market. We may choose $A_t\equiv t$ so that the process $B$ becomes
\begin{equation*}
B\equiv\left[\begin{array}{cc} 
1 & 0 \\
0& 0 \end{array}
\right]\!.
\end{equation*}
The sequence of (deterministic) constraint sets is defined by setting
\begin{equation*}
\U^n:=\{(x,y)\in\mb{R}^2 \,|\, y = nx,\, x\geq0\},\,\U:=\{(x,y)\in\mb{R}^2 \,|\, x = 0,\, y\geq0\}.
\end{equation*}
One can see that these cones are polyhedral and that we have $\{(0,0)\}=B\U\neq\Lim_{n\to+\infty}B\U^n=\mathbb{R}_+\times\{0\}$. Note though that we do have $\U=\Lim_{n\to+\infty}\U^n=\{0\}\times\mathbb{R}_+$. From this description we immediately have that in the limiting case the agent is only allowed to invest in stocks that do not yield any extra profit when compared to the bond while for $n\geq1$ they can choose an optimal strategy $\hat{\nu}^n$ and it does not matter that $\hat{\nu}^n_2=n\hat{\nu}^n_1$ may become arbitrarily large since it can be offset by a position in the bond, whose evolution is the same as that of $S^2$. Indeed, the optimal position in the first stock is $\hat{\nu}^n_1=\lambda_1/(1-p)$ which clearly does not converge to $0$, which is the only possible position in the first stock in the limiting case. The optimal wealth for $n\geq1$ is given by \[\hat{X}^n_t=x\,\exp\Biggl(\frac{\lambda_1}{1-p}\,W_t+\frac{\lambda_1^2(1-2p)}{2(1-p)^2}\,t\Biggr),\] which does not equal $\hat{X}\equiv x$, the optimal wealth process for the constraint set $\U$. Correspondingly, the value functions $u^n$ do not converge to $u$, since for $x>0$
\[u^n(x)=\frac1p\,x^p\exp\Biggl(\frac{p\lambda_1^2T}{2(1-p)}\Biggr)\quad\text{ and }\quad u(x)=\frac1p\,x^p.\]
\end{rmk}

\begin{rmk}
The reader may ask whether it is necessary to vary $\lambda$ \emph{and} $\mb{P}$ or whether by a sensible choice of the Girsanov transform this can be reduced to simply varying $\mb{P}$. In certain cases this is indeed the case, typically when $M=W$ is a Brownian motion. However in general not so as the following example illustrates. Set $M:=W\cdot W$ for a one dimensional Brownian motion. Thus the asset has dynamics 
\[dS_t=S_t(W_t\,dW_t+\lambda_t W_t^2\,dt)\quad\text{ under }\mb{P}.\] 
If $\lambda$ is allowed to vary, say to $\ti{\lambda}$, all models can be achieved such that 
\[dS_t=S_t(W_t\,dW_t+\ti{\lambda}_t W_t^2\,dt)\quad\text{ under }\mb{P}.\] 
However, if only $\mb{P}$ can be varied, we have $d\ti{\mb{P}}/d\mb{P}:=\mc{E}(-\beta\cdot W)$ and the process $S$ has dynamics 
\[dS_t=S_t(W_t\,d\ti{W}_t+(\lambda_t-\beta_t) W_t^2\,dt)\quad\text{ under }\ti{\mb{P}},\] 
where $\ti{W}$ is a $\ti{\mb{P}}$-Brownian motion. In particular we will find it impossible to recreate the first dynamics as $W$ is not a Brownian motion under $\ti{\mb{P}}$.
\end{rmk}

%%%%%%%%%%%%%%%%%%%%%%%%%%%%%%%%%%%%%%%%%%%%%%%%%%%%
%The Dual Domain in the presence of Cone Constraints
%%%%%%%%%%%%%%%%%%%%%%%%%%%%%%%%%%%%%%%%%%%%%%%%%%%%

\section{The Dual Domain in the Presence of Cone Constraints}
\label{sec_dualdom}
This section is devoted to a proof of Theorem \ref{thm_exist}, a full description of the dual domain. %in the
%case of cone constraints. Due to our choice of writing wealth in exponential format the proof becomes simpler when compared to \cite{KZ03} and \cite{LZ07}. \a{The theorem is slightly more general since it covers the ``multiplicative'' dual domain $\mc{Y}(1)$ which contains the ``additive'' dual domain from the cited references.} 
We note that Assumption \ref{ass_contfilt} is not required for any of the results in this section to hold, in particular dual elements need not be continuous, however the polyhedral nature of the cone cannot be dropped.
\begin{prop}\label{PropLZ07}
Let Assumption \ref{asscones} hold and $Y\in\mc{Y}(1)$ with $Y_T>0$. Then there exist:
\begin{enumerate}
\item A predictable $M$-integrable process $\kappa^Y$ with $B(\lambda-\kappa^{Y})\in (B\mc{K})^\circ$, $\mu^A$-a.e.
\item A local martingale $N^{Y}$ orthogonal to $M$.
\item A predictable
decreasing c{\`a}dl{\`a}g process $D^Y$ with $D^Y_0=1$ and $D^Y_T>0$ $\mb{P}$-a.s.
such that with the above
\begin{equation*}
Y=D^Y\,\mc{E}(-\kappa^Y\cdot M+N^Y).
\end{equation*}
\end{enumerate}
\end{prop}
\begin{proof}
Since $0\in\U$ $\mu^A$-a.e. we may proceed as in \cite{LZ07} Proposition 3.2. to deduce that a given $Y\in\mc{Y}(1)$ with $Y_T>0$ admits a multiplicative decomposition
which we can write as 
\begin{equation*}
Y=D^Y\,\mc{E}(-\kappa^Y\cdot M+N^Y),
\end{equation*}
where $D^Y$ is a positive, predictable, nonincreasing process with $D^Y_0=1$, $\kappa^Y$ is an $M$-integrable process and $N^Y$ a local martingale orthogonal to $M$. 
%see \cite{JS03} Lemma III.4.24. 
It thus remains to show that $B(\lambda-\kappa^Y)\in (B\mc{K})^\circ$ $\mu^A$-a.e. and we drop the superscripts in the remainder of the proof to ease the exposition. 

Set $F:=\log(D)$. By \cite{DS95} Theorem 2.1 there exists a predictable $\mu^A$-null set $E$ together with a nonnegative predictable process $\eta$ such that
\[F_t=-\int_0^t\eta_s\,dA_s+\int_0^t\I{E}(s)\,dF_s=:-\int_0^t\eta_s\,dA_s+F'_t.\]
From It\^{o}'s formula we derive
%\[dY_t=Y_{t-}\left(-\kappa_t\,dM_t+dN_t-\eta_t\,dA_t+dF_t'+d[F',N]_t\right),\]
that for any admissible investment strategy $\nu$ with corresponding wealth process $X^\nu\in\mc{X}(1)$ we have
%\begin{multline*}
%d(X_t^\nu Y_t)=X^\nu_{t}Y_{t-}\bigg(\big(\nu_t-\kappa_t\big)dM_t+dN_t+d[F',N]_t\\+\Big(\nu^\tr_tB_t^\tr B_t(\lambda_t-\kappa_t)-\beta_t\Big)\,dA_t+dF'_t\bigg).
%\end{multline*}
\begin{multline*}
d(X_t^\nu Y_t)=X^\nu_{t}Y_{t-}\bigg(\big(\nu_t-\kappa_t\big)^{\tr}dM_t+dN_t+d[F',N]_t+\Big(\nu^\tr_tB_t^\tr B_t(\lambda_t-\kappa_t)-\eta_t\Big)\,dA_t+dF'_t\bigg).
\end{multline*}
Observe that by Yoeurp's lemma $(\nu-\kappa\big)^{\tr}dM+dN+d[F',N]$ is the differential of a local martingale, $M$ being continuous. Since the product $X^\nu Y$ is a supermartingale, we hence must have that the differential 
\[ \Big(\nu^\tr B^\tr B(\lambda-\kappa)-\eta\Big)dA+dF'\] 
generates a nonpositive measure on the predictable $\sigma$-algebra $\mc{P}$. Since $\mu^A(E)=0$ we conclude, using the cone property of $\mc{K}$, that the following inequality must hold
%\begin{equation}
%\nu^\tr B^\tr B(\lambda-\kappa^Y)\leq\beta\label{eq_driftcond}
%\end{equation}
%$\mu^A$-a.e. for all $\nu\in\mc{A}_\U$. Using the cone property of $\U$ this implies
\begin{equation}
(B\nu)^\tr B(\lambda-\kappa)=\nu^\tr B^\tr B(\lambda-\kappa)\leq0\label{nulama0}
\end{equation}
$\mu^A$-a.e. for each $\nu\in\mc{A}_\U$. To conclude we have to show that arbitrary elements of $B\U$ can be realized as trading strategies, $\mu^A$-a.e. This is where the assumption that the constraints be polyhedral is needed.

Choosing $\nu = K^{1},\ldots,K^{m}$ it now follows that there exists a single $\mu^A$-null set (also denoted $E$)
such that for all $(t,\omega)\in E^c$ and all $j\in\{1,\ldots,m\}$
\begin{equation*}
\big(B_t(\omega)K^j_t(\omega)\big)^\tr B_t(\omega)\big(\lambda_t(\omega)-\kappa_t(\omega))\leq0.
\end{equation*}
In particular we have $B(\lambda-\kappa)\in (B\mc{K})^\circ$, $\mu^A$-a.e. as for
fixed $(t,\omega)$ any $k\in B_t(\omega)\U(t,\omega)$ may be written $\mu^A$-a.e. as
\begin{equation*}
k=\sum_{j=1}^m c_j B_t(\omega)K_t^j(\omega)
\end{equation*}
with some $c_j\geq0$ for $j\in\{1,\ldots,m\}$.
\end{proof}
\begin{rmk}
\label{rmk_K=Rd}
Suppose that $\mc{K}\equiv\mb{R}^d$ then in \eqref{nulama0}, given a $Y$ and corresponding $\kappa^Y$, we can directly insert $\nu=\lambda-\kappa^Y$. Integrating the resulting expression over $[0,T]$ with respect to $\mu^A$ we derive that
%\begin{align*}
%0\leq\E\left[\left\lo \big(\lambda -\kappa^{Y}\big)\cdot M,\big(\lambda -\kappa^{Y}\big)\cdot M\right\ro_T\right]=\E\left[\int_0^T\left(\lambda_t -\kappa^{Y}_t\right)^\tr B^\tr_t B_t\left(\lambda_t -\kappa^{Y}_t\right)dA_t\right]\leq0
%\end{align*}
%which implies that 
the stochastic integrals $\lambda\cdot M$ and $\kappa^{Y}\cdot M$ are indistinguishable and thus we recover the multidimensional version of \cite{LZ07} Proposition 3.2.
\end{rmk} 
%which states that \[Y=D^Y\,\mc{E}(-\lambda\cdot M+N^Y)\] for $Y\in\mc{Y}(1)$ with $Y_T>0$ when there are no constraints and where $D^Y$ and $N^Y$ are as above.
%\end{rmk}
%The following corollary is immediate.
\begin{cor}
\label{cor_kappa}
There exist a process $\hat{\kappa}$ and a local martingale $\hat{N}$ orthogonal to $M$, such that %\begin{equation*}
$\hat{Y}^{1}=\mc{E}\big(-\hat{\kappa}\cdot M+\hat{N}\big)$ for the dual optimizer $\hat{Y}^{1}$ where $y=1$ and
%\end{equation*}
$B(\lambda-\hat{\kappa})\in(B\mc{K})^\circ$ $\mu^A$-a.e. If $\hat{Y}^{y}$ denotes the dual optimizer for $y>0$ we  have that $\hat{Y}^{y}=y\hat{Y}^{1}=y\,\mc{E}\big(-\hat{\kappa}\cdot M+\hat{N}\big)$.
\end{cor}
\begin{proof} In view of Proposition \ref{PropLZ07} the key is to show that $\hat{Y}^1_T>0$ which is proved in Appendix \ref{existappend}. We may then 
%We first show that $\hat{Y}_T>0$ $\mb{P}$-a.s. In view of the relation $\hat{Y}_T=U'(\hat{X}_T)$ this is equivalent to the $\mb{P}$-a.s. finiteness of $\hat{X}_T$. Take $\mb{Q}\in\mc{M}^e(S)$
%and denote by $Y>0$ its density process with respect to
%$\mb{P}$. Then $\hat{X}$ is a nonnegative $\mb{Q}$-local
%martingale, hence a $\mb{Q}$-supermartingale. In
%particular,
%\[0\leq\E[\hat{X}_TY_T]=\E_{\mb{Q}}[\hat{X}_T]\leq x,\] from which it follows that $\hat{X}_T$ is finite $\mb{Q}$-a.s. hence also $\mb{P}$-a.s. as
%$\mb{Q}\sim\mb{P}$. We then may apply Theorem \ref{PropLZ07} and
proceed as in the proof of \cite{LZ07} Corollary 3.3. The independence of $y$ follows from the factorization property.
\end{proof}
%From the previous results we derive further properties of the processes $\hat{\nu}$ and $\hat{\kappa}$.
\begin{cor}
\label{cor_optkapandnu}
The optimal portfolio $\hat{\nu}$ satisfies $\mu^A$-a.e. for all admissible strategies $\nu$,
\begin{equation*}
\hat{\nu}^\tr B^\tr B\big(\lambda -\hat{\kappa}\big)=0 \quad \text{ and }\quad (\nu-\hat{\nu})^\tr B^\tr B(\lambda-\hat{\kappa})\leq 0.
\end{equation*}
\end{cor}
\begin{proof}
Due to factorization we may suppose that $x=1$. Then for the optimizers we know from Theorem \ref{thm_WZ09} (iii) that
the process $\hat{X}\hat{Y}^{y}$ is a martingale when $y=u'(1)$. We derive
\begin{equation*}
d(\hat{X}_t\hat{Y}^{y}_t)=\hat{X}_{t}\hat{Y}^{y}_{t-}\Big(\big(\hat{\nu}_t-\hat{\kappa}_t\big)^{\tr}dM_t+d\hat{N}_t+\Big(\hat{\nu}^\tr_tB_t^\tr B_t(\lambda_t-\hat{\kappa}_t)\Big)dA_t\Big).
\end{equation*}
Thanks to Assumption \ref{asscones} it must hold that $\hat{\nu}^\tr B^\tr B\big(\lambda -\hat{\kappa}\big)=0$ for all $\nu\in\mc{A}_\U,$ $\mu^A$-a.e. The second statement of the corollary now follows upon addition of \eqref{nulama0}.
\end{proof}

%%%%%%%%%%%%%%%%%%%%%%%%%%%%%%%%%%%%%%%%%%%%%%%%%%%%%%%%%%%%%%%%%%%%%%%%%
%Relationship with Quadratic Semimartingale BSDEs
%%%%%%%%%%%%%%%%%%%%%%%%%%%%%%%%%%%%%%%%%%%%%%%%%%%%%%%%%%%%%%%%%%%%%%%%%

\section{Relationship with Quadratic Semimartingale BSDEs}
\label{sec_linkBSDEs}
Having established a representation for elements of the dual domain, in this section we use this to connect the optimizers $(\hat{X},\hat{Y})$ with the solution triple of a specific BSDEs proving Proposition \ref{thm_bijection}. As noted before, admitting Theorem \ref{thm_exist}, one may find some of the results in \cite{Nu209} Corollaries 3.12 and 5.18, however we provide here a complete proof as it illustrates the interplay between $\hat{\kappa}$ and $\hat{\nu}$. Moreover, the verification argument is via uniqueness of BSDEs building on the following lemma whose proof we delegate to the appendix.% It relies on our \emph{special choice} of the dual domain which allows us to define the so-called dual opportunity process and to use its dynamic optimality properties.}We begin with an estimate on $\hat{\Psi}$.
\begin{lem}
In the setting of Theorem \ref{thm_WZ09} let $\hat{\Psi}:=\log\Bigl(\frac{u'(x)\hat{Y}^1}{U'(\hat{X})}\Bigr)$. Then $\hat{\Psi}\in\mf{E}$.
\end{lem}
We then derive
\begin{prop}
\label{thm_dualBSDE}
Under Assumptions \ref{ass_contfilt}, \ref{ass_expmom} and \ref{asscones} let $\hat{\nu}$ denote the optimal strategy, $\hat{X}$ the optimal wealth process and $\hat{Y}^1$ the optimal dual minimizer with decomposition $\hat{Y}^{1}=\mc{E}\big(\!-\hat{\kappa}\cdot M+\hat{N}\big)$.
If we set $\hat{\Psi}:=\log\bigl({u'(x)\hat{Y}^1/U'(\hat{X})}\bigr)$ and $\hat{Z}:=-\hat{\kappa}+(1-p)\hat{\nu}$, then the triple $(\hat{\Psi},\hat{Z},\hat{N})$ %\begin{equation*}
%:=\left(\!\log\left(\frac{u'(x)\hat{Y}^{1}}{U'(\hat{X})}\right),
%-\hat{\kappa}+(1-p)\hat{\nu},\hat{N}\right)
%\end{equation*}
is the unique solution to the BSDE \eqref{BSDE} with $\hat{\Psi}\in\mf{E}$ where
\begin{align*}
F(\cdot,z)=\frac{1}{2}\,\big\|Bz\big\|^2-\frac{q}{2}\,\big\|\Pi_{B\U}(B(z+\lambda))\big\|^2.
\end{align*}
\end{prop}
\begin{proof}
%\a{We observe that due to writing wealth in exponential format the representation of $\hat{\Psi}$ is very convenient} for. 
An application of It\^{o}'s formula to the process $\hat{\Psi}$ gives
\begin{align}
\label{eq_generator}
d\hat{\Psi}_t=\hat{Z}^{\tr}_t\,dM_t +d\hat{N}_t-\frac{1}{2}\,d\lo\hat{N},\hat{N}\ro_t
+\left[(1-p)\hat{\nu}_t^\tr B^\tr_t B_t\left(\lambda_t-\frac{\hat{\nu}_t}{2}\right)
-\frac{1}{2}\,\hat{\kappa}_t^\tr B_t^\tr B_t \hat{\kappa}_t\right]dA_t.
\end{align}
It remains to show that the generator in the previous equation corresponds to that given
in the statement of the theorem. Using the relation
\begin{equation*}
\hat{\nu}^\tr B^\tr B\lambda= \hat{\nu}^\tr B^\tr B\hat{\kappa}
\end{equation*}
implied by Corollary \ref{cor_optkapandnu} we end up with the following form
for the generator of \eqref{eq_generator},
\begin{align*}
\frac{1}{2}\|B\hat{Z}\|^2+
\frac{p(1-p)}{2}\left\|B\!\left(\frac{\hat{Z}+\lambda}{1-p}\right)\right\|^2
-\frac{p(1-p)}{2}\left\|B\!\left(\hat{\nu}-\frac{\hat{Z}+\lambda}{1-p}\right)\right\|^2.
\end{align*}
Now from the definition of $\hat{Z}$ together with Corollary \ref{cor_optkapandnu} the following equation holds $\mu^A$-a.e. for all admissible
$\nu$
\begin{equation*}
(\nu-\hat{\nu})^\tr B^\tr B\left[(1-p)\hat{\nu}-(\hat{Z}+\lambda)\right]\geq0.
\end{equation*}
This equation can be understood as the subgradient condition for the convex function
\begin{equation*}
\R^d\ni\eta\mapsto\frac{1-p}{2}\,\left\|B\!\left(\eta-\frac{\hat{Z}+\lambda}{1-p}\right)\right\|^2
\end{equation*}
to have a minimum over $\mc{K}$ at $\hat{\nu}$ holding $\mu^A$-a.e. In particular, $\mu^A\text{-a.e.}$
\[ \frac{1-p}{2}\,\left\|B\!\left(\hat{\nu}-\frac{\hat{Z}+\lambda}{1-p}\right)\right\|^2=\frac{1-p}{2}\,\inf_{\eta\in\mc{K}}\left\|B\!\left(\eta-\frac{\hat{Z}+\lambda}{1-p}\right)\right\|^2.\]
Since it coincides with the generator of \eqref{eq_generator} $\mu^A$-a.e. $F$ is hence of the claimed form, paying attention to the signs and using the Pythagorean rule (see Theorem \ref{thm_Moreau}).

As the filtration is continuous, $\hat{N}$ is continuous and we have constructed a solution $(\hat{\Psi},\hat{Z},\hat{N})$ to the BSDE \eqref{BSDE} with $\hat{\Psi}\in\mf{E}$. The claimed uniqueness follows then from Theorem \ref{thmMW10}
noting that Proposition \ref{prop_Fnproperties} implies the required Assumption \ref{assBSDE}.
\end{proof}

To write the processes $\hat{X}$ and $\hat{Y}$ in terms of the solution to
the above BSDE $(\hat{\Psi},\hat{Z},\hat{N})$ we first recall a classical result.
\begin{thm}[Moreau Orthogonal Decomposition]
\label{thm_Moreau}
Let $Q\subset\mb{R}^d$ be a closed convex cone and $Q^\circ\subseteq\mb{R}^d$ its polar cone.
Then for all $q,r,u\in\R^d$ the following statements are equivalent:
\begin{enumerate}
\item $u=q+r, \,q\in Q,\, r\in Q^\circ\text{ and } q^\tr r  =0$,
\item $q=\Pi_Q(u),\,r=\Pi_{Q^\circ}(u)$,
\end{enumerate}
where $\Pi$ denotes the projection or nearest point operator onto the indicated set.
\end{thm}
\begin{prop}
Suppose that $y=u'(x)$ for some $x>0$ and that the Assumptions \ref{ass_contfilt}, \ref{ass_expmom} and \ref{asscones} hold. 
Given $(\hat{\Psi},\hat{Z},\hat{N})$, the unique solution to the BSDE \eqref{BSDE} with $\hat{\Psi}\in\mf{E}$ and the above driver
$F$ we can write the optimizers, up to indistinguishability, as 
\begin{align*}
\hat{X}^x=x\,\mc{E}(\ti{\nu}\cdot M +\ti{\nu}\cdot \lo M,M\ro\lambda),\quad \quad
\hat{Y}^{y}=y\,\mc{E}\big(-\ti{\kappa}\cdot M+\hat{N}\big),
\end{align*}
where the predictable integrands $\ti{\nu}$ and $\ti{\kappa}$
are defined via
\begin{align*}
\ti{\nu}:=\frac{1}{1-p}P^\tr\ti{\Gamma}^{\frac12}\left[\Pi_{B\mc{K}}\Bigl(B(\hat{Z}+
\lambda)\Bigr)\right],\quad
\ti{\kappa}:=
P^\tr\ti{\Gamma}^{\frac12}\left[B\lambda-\Pi_{(B\mc{K})^\circ}\Bigl(B(\hat{Z}+\lambda)\Bigr)\right]
\end{align*}
and satisfy, $\mu^A$-a.e. $B\ti{\nu}=B\hat{\nu}\text{ and }B\ti{\kappa}=B\hat{\kappa}$.
The process $(\ti{\Gamma}^{i,j})_{i,j=1,\ldots,d}$ is chosen to be a predictable process valued in the space of $d\times d$ diagonal matrices such that 
\begin{equation*}
\ti{\Gamma}^{ij}=
\begin{cases}
1\big/\Gamma^{ii}& \text{ if } i=j \text{ and }\Gamma^{ii}\neq0\\
0 & \text{ if }i\neq j.
\end{cases}
\end{equation*}
\end{prop}
\begin{proof}
The formulae for $\hat{X}$ and $\hat{\nu}$ are given (up to null-investments) in \cite{Nu209} Corollary 3.12, cf. also \cite{HIM05} Theorem 14 and \cite{Mo09} Theorem 4.4. To derive the result for $\hat{Y}$ observe that from Proposition \ref{thm_dualBSDE} and the uniqueness result in Theorem \ref{thmMW10} (ii) we have the relation $\hat{Z}\equiv-\hat{\kappa}+(1-p)\hat{\nu}$ which is equivalent to
\begin{equation*}
B(\hat{Z}+\lambda)=B(\lambda-\hat{\kappa})+(1-p)B\hat{\nu}.
\end{equation*}
Since $\mc{K}$ is a cone we see that $(1-p)B\hat{\nu} \in B\mc{K}$, combining this with Corollary \ref{cor_optkapandnu} and using Theorem \ref{thm_Moreau}
we deduce that up to a $\mu^A$-null set
\begin{align}
\label{eq_optkap}
(1-p)B\hat{\nu}& = \Pi_{B\mc{K}}\!\left(B(\hat{Z}+\lambda)\right)\quad\text{ and }\quad B(\lambda-\hat{\kappa})=\Pi_{(B\mc{K})^\circ}\!\left(B(\hat{Z}+\lambda)\right).
\end{align}
%Thus uniqueness for $\hat{Z}$
%given by the comparison theorem implies uniqueness
%for the processes $B\hat{\kappa}$ and $B\hat{\nu}$. 
We then use the relation $B=\Gamma^{1/2}P$ to write
\begin{equation*}
\Gamma^{1/2}P\hat{\kappa}=B\lambda-\Pi_{(B\mc{K})^\circ}\!\left(B(\hat{Z}+\lambda)\right).
\end{equation*}
The matrix valued process $\Gamma$ may have some zero diagonal elements and so we may not be able to invert the above relation uniquely. However, by the construction of the process $\ti{\kappa}$
we have that $B\hat{\kappa}=B\ti{\kappa}$ holds $\mu^A$-a.e. Integrating
the difference over $[0,T]\times\Omega$ with respect to $\mu^A$ shows
\begin{align*}
\E\!\left[\left\lo \big(\hat{\kappa} -\ti{\kappa}\big)\cdot M,\big(\hat{\kappa} -\ti{\kappa}\big)\cdot M\right\ro_T\right]%=\E\!\left[\int_0^T\|B_t\big(\hat{\nu}_t -\ti{\nu}_t\big)\|^2\,dA_t\right]\\
=\int_{[0,T]\times\Omega}\|B\big(\hat{\kappa} -\ti{\kappa}\big)\|^2\,d\mu^A=0.
\end{align*}
In particular the stochastic integrals $\hat{\kappa}\cdot M$ and $\ti{\kappa}\cdot
M$ are indistinguishable so that the representation for $\hat{Y}$ now follows. 
\end{proof}
%\a{With regards to the non unique representation of the optimal strategy we refer to the Remarks \ref{rmk_nullinv} and \ref{rmk_KK}.}
%%%%%%%%%%%%%%%%%%%%%%%%%%%%%%%%%%%%%%%%%%%%%%%%%%%%%%%%%%%%%%%%%%%%%%%%%%%%%%%%%%%%%%%%%%%%%%
%% Continuity of the Optimizers
%%%%%%%%%%%%%%%%%%%%%%%%%%%%%%%%%%%%%%%%%%%%%%%%%%%%%%%%%%%%%%%%%%%%%%%%%%%%%%%%%%%%%%%%%%%%%%#

\section{Continuity of the Optimizers}
\label{sec_stability}
In this section we prove Theorem \ref{thm_mainstab} on the continuity of the optimizers 
\begin{equation*}
\hat{X}^{n}:=\hat{X}(\lambda^{n},\mb{P}^{n},p^{n},\U^{n})\text{ and } \hat{\nu}^{n}:=\hat{\nu}(\lambda^{n},\mb{P}^{n},p^{n},\U^{n}), 
\end{equation*}
for the problem 
\begin{equation*}
u^{n}(x):=\sup_{\nu\in\mc{A}_{\U^{n}}}\,\E_{\mb{P}^{n}}\!\!\left[U^{n}\big(X^{{n},x,\nu}_T\big)\right]
\end{equation*}
discussed in Section \ref{sec_sensitivity}, to which we refer for any unexplained notation. We assume throughout that Assumptions \ref{ass_contfilt}, \ref{ass_unif.expmom} and \ref{ass_prefconv} hold and that $x=1$ which, due to the factorization property, is no loss of generality. The first result is a consequence of the standing assumptions
which is used repeatedly and whose proof is left to the reader.
\begin{lem}
\label{lem_convterminal}
The sequence of random variables $(\zeta^{n})_{{n}\in\mb{N}}$ defined via 
\begin{equation*}
\zeta^n:=(L^n)^*+\lo L^n, L^n\ro_T,  
\end{equation*}
converges to zero in $\mb{P}$-probability and satisfies $\sup_{n\in\mb{N}}\E[\exp(c\zeta^{n})]<+\infty$ for all $c>0$.
\end{lem}

Given the optimizers $(\hat{X}^{n},\hat{Y}^{n})$ Proposition \ref{thm_dualBSDE} describes the link
to the solution triple $(\hat{\Psi}^{n},\hat{Z}^{n},\hat{N}^{n})$ of the following BSDE under $\mb{P}^{n}$ 
for $n\in\mb{N}_0$ (written in generic variables $(\Psi,Z,N)$),
\begin{equation}
\label{eq_BSDEPn}
d\Psi_t=Z_t^\tr\,dM^n_t+dN_t-F_1^{n}(t,Z_t)\,dA_t-\frac{1}{2}\,d\lo N,N\ro_t, \quad \Psi_T=0.
\end{equation}
Here 
\begin{gather*}
F_1^{n}(\cdot,z)=\frac{1}{2}\,\big\|Bz\big\|^2-\frac{q^{n}}{2}\,\big\|\Pi_{B\U^{n}}(B(z+\lambda^{n}-\beta^{n}))\big\|^2,
\end{gather*}
$M^n:=M+\lo M,M\ro\beta^n$ and $N$ are $\mb{P}^n$-local martingales which are orthogonal
and the necessary integrability conditions are satisfied with respect to the measure $\mb{P}^n$.
To deduce the convergence of $(\hat{X}^n,\hat{Y}^n)$ we shall show first that $(\hat{\Psi}^{n},\hat{Z}^{n},\hat{N}^{n})$ converges to $(\hat{\Psi},\hat{Z},\hat{N})$. In order to do this it is necessary to perform a change of variables
related to considering the BSDE \eqref{eq_BSDEPn} under $\mb{P}$ rather than $\mb{P}^n$. This is the content of the next proposition.
\begin{prop}
\label{prop_BSDEP}
Let $(\hat{\Psi}^{n},\hat{Z}^{n},\hat{N}^{n})$ be as above then the triple 
\begin{equation*}
(\hat{\Xi}^n,\hat{V}^{n},\hat{O}^{n}):=
\left(\hat{\Psi}^{n}+L^{n}-\tfrac{1}{2}\lo L^{n}, L^{n}\ro,\hat{Z}^{n},\hat{N}^{n}+\lo \hat{N}^{n},L^{n}\ro +L^{n}\right)
\end{equation*}
is the unique solution to the BSDE under $\mb{P}$
\begin{equation}
\label{BSDEP1}
d\Psi_t=Z_t^\tr\,dM_t+dN_t-F^{n}(t,Z_t)\,dA_t-\frac{1}{2}\,d\lo N,N\ro_t, \quad \Psi_T=L^{n}_T-\tfrac{1}{2}\lo L^{n}, L^{n}\ro_T,
\end{equation}
with $\Psi\in\mathfrak{E}$ where the generator is given by 
\begin{gather}
\label{eq_defng}
F^{n}(\cdot,z)=\frac{1}{2}\,\big\|Bz\big\|^2-\frac{q^{n}}{2}\,\big\|\Pi_{B\U^{n}}(B(z+\lambda^{n}-\beta^{n}))\big\|^2
-(Bz)^\tr (B\beta^{n}),
\end{gather}
$q^{n}$ is the dual number corresponding to $p^{n}$ and the process $N$ is a $\mb{P}$-local martingale orthogonal to $M$.
%processes $N^n$ and %$N^{n}=\hat{N}^{n}+\lo \hat{N}^{n},L^{n}\ro$ and 
%$L^{n}$ are $\mb{P}$-local martingales, both orthogonal to $M$. 
%The mapping taking $(\hat{\Psi}^n,\hat{Z}^{n},\hat{N}^{n})$ to
%$(\hat{\Psi}^n,\hat{Z}^{n},\hat{O}^{n})$ provides a bijection between the two BSDEs.
\end{prop}
\begin{proof}
The Girsanov theorem implies that $\hat{O}^n$ is a $\mb{P}$-local martingale and its orthogonality to $M$ follows from the 
fact that $\lo\hat{N}^n,M^n\ro\equiv0$ and  $\lo M ,L^n\ro\equiv0$. %An application of It\^{o}'s formula shows that 
Thanks to \eqref{eq_BSDEPn} the triple $(\hat{\Xi}^n,\hat{V}^{n},\hat{O}^{n})$ then solves \eqref{BSDEP1} with driver \eqref{eq_defng}. Moreover, once we show that $\hat{\Xi}^n\in\mathfrak{E}$ then Theorem \ref{thmMW10} (ii) provides the claimed uniqueness. 
Via H\"{o}lder's inequality, using the notation of Lemma \ref{lem_convterminal}, we have the estimate 
\begin{equation*}
\E\Bigl[\exp\big(c(\hat{\Xi}^n)^*\big)\Bigr]\leq \E\!\left[\left(\frac{d\mb{P}}{d\mb{P}^n}\right)^2\right]^{1/2}\E_{\mb{P}^n}\!\left[\exp\!\big(4c(\hat{\Psi}^n)^*\big)\right]^{1/2}
+\E[\exp(2c\zeta^{n})]<+\infty
\end{equation*}
for all $c>0$. This completes the proof.
\end{proof}

The BSDE (under $\mb{P}=\mb{P}^0$) satisfied by $(\hat{\Psi},\hat{Z},\hat{N})=(\hat{\Psi}^{0},\hat{Z}^{0},\hat{N}^{0})$  related to the optimizers $(\hat{X},\hat{Y})=(\hat{X}^{0},\hat{Y}^{0})$ is given by
\begin{gather}\label{BSDEF0}
d\Psi_t=Z^{\tr}_t\,dM_t+dN_t-F(t,Z_t)\,dA_t-\frac{1}{2}\,d\lo N,N\ro_t,\quad\Psi_T=0,
\end{gather}
where the driver $F=F^{0}$ satisfies  
\begin{align*}
F(\cdot,z)=\frac{1}{2}\,\big\|Bz\big\|^2-
\frac{q}{2}\,\big\|\Pi_{B\U}(B(z+\lambda))\big\|^2.
\end{align*}
Our goal is continuity of the optimizers, which we prove via the stability result in Theorem \ref{ThmStab}. 
We show that it implies convergence of $(\hat{\Xi}^{n},\hat{V}^{n},\hat{O}^{n})$ to $(\hat{\Psi},\hat{Z},\hat{N})$ in an appropriate sense
and then deduce the result for $(\hat{\Psi}^{n},\hat{Z}^{n},\hat{N}^{n})$. We first collect some properties of the drivers $F^{n}$. 

\begin{prop} For each ${n}\in\mb{N}_0$,
\label{prop_Fnproperties}
\begin{enumerate}
        \item The driver $F^{n}$ is continuously differentiable and convex in $z$.
        \item It satisfies a quadratic growth condition in $z$. More precisely,
\begin{equation*}        
|F^{n}(t,z)|\leq\frac12\,\|B_t\beta^{n}_t\|^2+|q^{n}|\|B_t(\lambda^{n}_t-\beta_t^{n})\|^2+(1+|q^{n}|)\|B_tz\|^2.
\end{equation*}
        \item The function $F^{n}$ is locally Lipschitz continuous in $z$, i.e. for all $z_1,z_2\in\mb{R}^d$
\begin{multline*}
|F^{n}(t,z_1)-F^{n}(t,z_2)|\\\leq (1+|q^{n}|)\Big(\|B_t\beta^{n}_t\|+\|B_tz_1\|
+\|B_tz_2\|+\|B_t(\lambda^{n}_t-\beta_t^{n})\|\Big)\big\|B_t(z_1-z_2)\big\|.
\end{multline*}
\item Under the Assumptions \ref{ass_contfilt}, \ref{ass_unif.expmom} and \ref{ass_prefconv} the drivers converge in the sense that 
\begin{equation*}
\lim_{{n}\to+\infty}\int_0^T|F^{n}(t,\hat{Z}_t)-F(t,\hat{Z}_t)|\,dA_t=0
\end{equation*}
in $L^1(\mb{P})$ and hence in $\mb{P}$-probability, where $\hat{Z}$ is the process from above. 
\end{enumerate}
\end{prop}
\begin{proof}
Items (ii) and (iii) follow from the explicit form of the driver together with the Lipschitz property of the distance
function. Items (i) and (iv) are a little more involved and we provide a proof, suppressing the argument $(t,\omega)$ for brevity. Starting with item (i) we recall from Borwein and Lewis \cite{BL06} Section 3.3 that for the function $\theta:\mb{R}^d\to \mb{R}$, 
\[\theta(z):=\big\|Bz-\Pi_{B\U}(Bz)\big\|^2,\] 
we have 
 \[D_z \theta(z_0)(\cdot)=2\,\big\lo Bz_0-\Pi_{B\U}(Bz_0),B(\cdot)\big\ro \]
where $D_z \theta(z_0)$ denotes the differential of $\theta$ with respect to $z$ at a point $z_0\in\mb{R}^d$ (a linear functional on $\mb{R}^d$) and $\lo\cdot,\cdot\ro$ stands for the inner product on $\mb{R}^d$.

From Theorem \ref{thm_Moreau}, we can show differentiability of $F^n$, indeed
\begin{equation*}
\big\|B(z+\lambda^{n}-\beta^{n})\big\|^2=\big\|B(z+\lambda^{n}-\beta^{n})-\Pi_{B\U^{n}}\!(B(z+\lambda^{n}-\beta^{n}))\big\|^2+\big\|\Pi_{B\U^{n}}(B(z+\lambda^{n}-\beta^{n}))\big\|^2
\end{equation*}
and we conclude that
%\begin{equation*}
$D_z F^{n}(z_0)(\cdot)=\big\lo Bz_0-q^{n}\Pi_{B\U^{n}}(B(z_0+\lambda^{n}-\beta^{n}))-B\beta^{n},B(\cdot)\big\ro$.
%\end{equation*}

As to convexity we then derive from the Lipschitz property of $\Pi_{B\U^{n}}$ and the Cauchy-Schwarz inequality, that for $q^{n}\in(0,1)$ and for all $z_1,z_2\in\mb{R}^d$
\begin{align*}
\Big(D_z F^{n}(z_1)-D_z F^{n}(z_2)\Big)(z_1-z_2)\geq (1-q^{n})\big\|B(z_1-z_2)\big\|^2\geq 0.
\end{align*}
This is the multidimensional version of monotonicity of the derivatives and it is equivalent to the convexity property, see \cite{BL06} Section 3.1.
For $q^{n}\in(-\infty,0)$ we use the representation
\begin{align*}
F^{n}(z)
%\frac{1}{2}\,\big\|Bz\big\|^2-\frac{q^{n}}{2}\,\big\|B(z+\lambda^{n}-\beta^{n})\big\|^2-(Bz)^\tr (B\beta^{n})+\frac{q^{n}}{2}\inf_{\eta\in B\U^{n}}\,\big\|\eta-B(z+\lambda^{n}-\beta^{n})\big\|^2\\
&=\frac{1}{2}\,\big\|Bz\big\|^2-(Bz)^\tr (B\beta^{n})+\frac{q^{n}}{2}\inf_{\eta\in B\U^{n}}\left(\|\eta\|^2-2\big\lo\eta,B(z+\lambda^{n}-\beta^{n})\big\ro\right).
\end{align*}
An infimum of affine functions (in $z$) is concave (in $z$), hence the last term is convex in $z$ due to the sign of $q^{n}$. Thus $F^{n}$ is convex as a sum of two convex functions.
 
We continue with item (iv). Using the definition of the drivers one can derive the following inequality
\begin{align*}
|F^{n}(t,\hat{Z}_t)-F(t,\hat{Z}_t)|\leq&\frac{|q|}{2}\cdot\left|\left\|\Pi_{B_t\U_t^{n}}\big(B_t(\hat{Z}_t+\lambda_t)\big)\right\|^2 -\left\|\Pi_{B_t\U_t}\big(B_t(\hat{Z}_t+\lambda_t)\big)\right\|^2\right|\\
&+ \|B_t\hat{Z}_t\|\cdot\|B_t\beta^{n}_t\|\\&
+\left|\frac{q-q^{n}}{2}\right|\cdot\left\|\Pi_{B_t\U_t^{n}}\big(B_t(\hat{Z}_t+\lambda_t)\big)\right\|^2\\
&+\frac{|q^{n}|}{2}\cdot\left|\left\|\Pi_{B_t\U_t^{n}}\big(B_t(\hat{Z}_t+\lambda_t)\big)\right\|^2
-\left\|\Pi_{B_t\U_t^{n}}\big(B_t(\hat{Z}_t+\lambda^{n}_t-\beta^{n}_t)\big)\right\|^2\right|\\
&=:G^{n}_t+H^{n}_t+I^{{n}}_t+J^{n}_t.
\end{align*}
We have to show that 
\begin{equation*}
%\lim_{n\to+\infty}\int_0^T G_t^{n}\,dA_t = 0 \text{  } \mb{P}\text{-a.s. and }
\lim_{{n}\to+\infty}\E\!\left[\int_0^T(G_t^{n}+H^{n}_t +I^{n}_t+J^{n}_t)\,dA_t\right]=0,
\end{equation*}
for which we work term by term, beginning with $G^{n}$. %, since $\mc{K}$ is polyhedral we have in all cases that 
%\begin{equation*}
%%\lim_{n\to+\infty}\int_0^T G_t^{n}\,dA_t = 0 \text{  } \mb{P}\text{-a.s. and }
%B\mc{K}=\Lim_{{n}\to+\infty}B\mc{K}^n,\quad \mu^A {a.e.}
%\end{equation*}
By Proposition \ref{kuraconv} $(G^{n})_{n\in\mathbb{N}}$ then converges to zero $\mu^A-a.e.$ and is dominated by $|q|\cdot\|B(\hat{Z}+\lambda)\|^2$. In particular thanks to the dominated convergence theorem we have 
\begin{equation*}
\lim_{n\to+\infty} \E\!\left[\int_0^TG^{n}_t\,dA_t\right] = \lim_{n\to+\infty} \int_{[0,T]\times\Omega}G^{n}\,d\mu^A  = 0.
\end{equation*}
%as required.
For the second term we apply the Cauchy-Schwarz inequality to get
\begin{equation*}
\E\!\left[\int_0^T H^{n}_t\,dA_t\right]^2\leq \E\big[\big\lo \hat{Z}\cdot M,\hat{Z}\cdot
M\big\ro_T\big]\,\E\big[\big\lo \beta^{n}\cdot M,\beta^{n}\cdot
M\big\ro_T\big].
\end{equation*}
The convergence to zero now follows from Assumption \ref{ass_unif.expmom}
and the condition on $\beta^{n}$. For the $I^{{n}}$ terms we apply the contraction
property of the projection map to deduce
\begin{equation*}
\E\!\left[\int_0^T I^{{n}}_t\,dA_t\right]\leq \frac{|q-q^{n}|}{2}\,\E\!\left[\left\lo (\hat{Z}+\lambda)\cdot M,(\hat{Z}+\lambda)\cdot M\right\ro_T\right],
\end{equation*}
from which the convergence follows. For the final term we first derive, similarly
to item (iii), the local Lipschitz estimate
\begin{equation*}
J_t^{n} \leq |q^{n}|\Big(2\|B_t\hat{Z}_t\|+\|B_t\lambda_t\|
+\|B_t\lambda^{n}_t\|+\|B_t\beta_t^{n}\|\Big)
\big\|B_t(\lambda_t-\lambda^{n}_t+\beta^{n}_t)\big\|.
\end{equation*}
Applying the Cauchy-Schwarz and Young inequalities we derive the existence
of a constant $\hat{c}$, independent of ${n}$ (due to the convergence assumptions the sequences appearing in the estimates are bounded), such that
\begin{equation*}
\E\!\left[\int_0^T J^{n}_t\,dA_t\right]^2\leq \hat{c}\,\E\!\left[\Big\lo (\lambda-\lambda^{n}+\beta^{n})\cdot M,(\lambda-\lambda^{n}+\beta^{n})\cdot M\Big\ro_T\right].
\end{equation*}
Letting $n$ go to infinity and using Assumptions \ref{ass_unif.expmom} and
\ref{ass_prefconv} the result follows.
\end{proof}  

\begin{thm}
\label{thm_stability}
Let the triple $(\hat{\Psi}^{n},\hat{Z}^{n},\hat{N}^{n})$ denote the
unique solution to the BSDE \eqref{eq_BSDEPn} then 
\begin{gather*}
\lim_{{n}\to+\infty}\E\!\left[\exp\!\left(\rho\,\big(\hat{\Psi}^{n}-\hat{\Psi}\big)^*\right)\right]=1,\\
\lim_{{n}\to+\infty}\E\!\left[\left( \big\lo(\hat{Z}^{n}-\hat{Z})\cdot M,(\hat{Z}^{n}-\hat{Z})\cdot M\big\ro_T
+\lo \hat{N}^{n}-\hat{N},\hat{N}^{n}-\hat{N}\ro_T\right)^{\rho/2}\right]=0,
\end{gather*}
for all $\rho\geq1$, where $(\hat{\Psi},\hat{Z},\hat{N})$ denotes the unique solution triple of the BSDE \eqref{BSDE} with $\hat{\Psi}\in\mf{E}$. 
%and the expectation is with respect to $\mb{P}$. 
\end{thm}
\begin{proof}
Using the notation of Lemma \ref{lem_convterminal} and Proposition \ref{prop_BSDEP} we can write
\[0\leq \big(\hat{\Psi}^n-\hat{\Psi}\big)^*\leq\big(\hat{\Xi}^n-\hat{\Psi}\big)^*+(\zeta^n)^*.\]
Hence the sequence $\left(\exp\!\Big(\rho\big(\hat{\Psi}^n-\hat{\Psi}\big)^*\Big)\right)_{n\in\mb{N}}$ is uniformly integrable and converges to zero in $\mb{P}$-probability. Both these claims are consequences of Lemma \ref{lem_convterminal} and Theorem \ref{ThmStab}, whose conditions are guaranteed by Proposition \ref{prop_Fnproperties} and Assumption \ref{ass_unif.expmom}. Since $\hat{Z}^n\equiv\hat{V}^n$ and 
\[ \lo \hat{N}^{n}-\hat{N},\hat{N}^{n}-\hat{N}\ro_T\leq 2\lo \hat{O}^{n}-\hat{N},\hat{O}^{n}-\hat{N}\ro_T 
+2\lo L^{n},L^{n}\ro_T,\]
we derive the second convergence in a similar fashion.
%Using the notation of Proposition \ref{prop_BSDEP} and Theorem \ref{ThmStab}, whose conditions are guaranteed by 
%Proposition \ref{prop_Fnproperties} and Assumption \ref{ass_unif.expmom}, we deduce
%\begin{gather*}
%\lim_{{n}\to+\infty}\E\!\left[\exp\!\left(\rho\left(\hat{\Xi}^{n}-\hat{\Psi}\right)^*\right)\right]=1,\\
%\lim_{{n}\to+\infty}\E\!\left[\Big(\big\lo (\hat{V}^{n}-\hat{Z})\cdot M,(\hat{V}^{n}-\hat{Z})\cdot M\big\ro_T
%+\lo \hat{O}^{n}-\hat{N},\hat{O}^{n}-\hat{N}\ro_T\Big)^{\rho/2}\right]=0,
%\end{gather*}
%for all $\rho\geq1$. 
%Since $\hat{Z}^n=\hat{V}^n$ this implies the claimed result for 
%$\hat{Z}^n\cdot M$. Using Proposition \ref{prop_BSDEP} we have the estimates
%\begin{align*}
%\lo \hat{N}^{n}-\hat{N},\hat{N}^{n}-\hat{N}\ro_T\leq 2\lo \hat{O}^{n}-\hat{N},\hat{O}^{n}-\hat{N}\ro_T 
%+2\lo L^{n},L^{n}\ro_T,\\
%\exp\!\left(\rho\left(\hat{\Psi}^{n}-\hat{\Psi}\right)^*\right)
%\leq \exp\!\left(2\rho\left(\hat{\Xi}^{n}-\hat{\Psi}\right)^*\right)
%+ \exp\!(2\rho\zeta^n)
%\end{align*} 
%from which the claimed convergence holds thanks to Lemma  \ref{lem_convterminal}
%and the stability result above. 
\end{proof}
We now show how this implies convergence of the objects of interest, and begin with the primal variables. 
\begin{thm}
\label{thm_nuconv}
We have that for all $\rho\geq1$
\begin{equation*}
\lim_{n\to+\infty}\E\left[\big\lo(\hat{\nu}^{n}-\hat{\nu})\cdot M,(\hat{\nu}^{n}-\hat{\nu})\cdot M\big\ro_T^{\rho/2}\right] = 0.
\end{equation*}
In particular, $(\hat{\nu}^{{n}}-\hat{\nu})\cdot M$ converges to zero in $\mc{M}^2$ and hence in the semimartingale topology. 
\end{thm}
\begin{proof}
Using the definitions, it follows that 
\begin{equation*}
\lim_{\n\to+\infty}\E\left[\big\lo(\hat{\nu}^{n}-\hat{\nu})\cdot M,(\hat{\nu}^{n}-\hat{\nu})\cdot M\big\ro_T^{\rho/2}\right] = 0
\end{equation*}
is equivalent to 
\begin{equation*}
\lim_{n\to+\infty}\E\!\left[\!\left(\!\int_0^T\!\left\|\frac{\Pi_{B_t\mc{K}_t^{n}}\big(B_t(\hat{Z}_t^{n}+\lambda_t^{n}-\beta^{n}_t)\big)}{(1-p^{n})}
-\frac{\Pi_{B_t\mc{K}_t}\big(B_t(\hat{Z}_t+\lambda_t)\big)}{(1-p)}\right\|^2 \!dA_t\!\right)^{{\rho}/{2}}\right]=0.
\end{equation*}
To establish this we proceed similarly to the proof of Proposition \ref{prop_Fnproperties} (iv) 
so that Proposition \ref{semimarttop} (i) then yields the assertion.%An application
\end{proof}
\begin{thm}
\label{thm_Xconv}
The sequence of processes $\hat{X}^n\in\mc{X}(x)$, $n\in\mb{N}$, converges to $\hat{X}\equiv\hat{X}^0\in\mc{X}(x)$ in the semimartingale topology.
\end{thm}
\begin{proof}
We note the dynamics of the optimal wealth processes given by \eqref{eq_wealthUnderPn} and
set 
\begin{equation*}
\Upsilon^n:=\hat{\nu}^n\cdot M+\hat{\nu}^n\cdot\lo M,M\ro\lambda^n, 
\end{equation*}
for $n\in\mb{N}_0$. We show the convergence in $\mc{H}^2$ of the sequence $(\Upsilon^n)_{n\in\mb{N}}$ so that the result of the theorem will follow via Proposition \ref{semimarttop} (ii) since $\hat{X}^n=\mc{E}(\Upsilon^n)$ and $\hat{X}=\mc{E}(\Upsilon^0)$. Observe from Theorem \ref{thm_nuconv} that $(\hat{\nu}^{{n}}-\hat{\nu})\cdot M$ converges to zero in $\mc{M}^2$ so that we need only show the convergence of the finite variation parts, namely that 
\begin{equation*}
\lim_{n\to+\infty}\E\!\left[\left(\int_0^T\big|\,d\big(\lo \hat{\nu}^n\cdot M,\lambda^n\cdot M\ro-\lo \hat{\nu}\cdot M,\lambda\cdot M\ro\big)\big|\right)^2\right]=0.
\end{equation*}
Adding and subtracting $\lo \hat{\nu}\cdot M, \lambda^n\cdot M\ro$ and then applying the Kunita-Watanabe inequality, we see that the
above holds due to Theorem \ref{thm_nuconv} together with the convergence of $\lo(\lambda^n-\lambda)\cdot M,(\lambda^n-\lambda)\cdot M\ro_T$ to zero in all $L^\rho(\mb{P})$ spaces. 
\end{proof}
\begin{thm}
\label{thm_uconv}
The value functions $u^{n}$ converge pointwise to $u$. Their derivatives converge pointwise to $u'$.
\end{thm}
\begin{proof}
From the BSDE \eqref{BSDEF0} the reader may verify the relation,
\begin{equation*}
d(\exp(\hat{\Psi})U'(\hat{X}))_t=\exp(\hat{\Psi}_t)U'(\hat{X}_t)(-\hat{\kappa}_t\,dM_t+d\hat{N}_t)
\end{equation*}
which implies that
\begin{equation*}
\hat{Y}=u'(x)\hat{Y}^1=\exp(\hat{\Psi})U'(\hat{X})=e^{\hat{\Psi}_0}x^{p-1}\hat{Y}^1 \quad\mb{P}\text{-a.s.}
\end{equation*}
It then follows that $c_p=e^{\hat{\Psi}_0}$ $\mb{P}$-a.s.
which shows that 
\begin{equation*}
u^{n}(x)=U^{n}(x)c^{n}_{p^{n}}=U^{n}(x)e^{{\hat{\Psi}^{n}}_0} \quad\text{$\mb{P}$-a.s.}
\end{equation*}
From Theorem \ref{thm_stability} we have that
%\begin{equation*}
$\lim_{{n}\to+\infty}|\hat{\Psi}^{n}_0-\hat{\Psi}_0|=0$
%\end{equation*}
in probability. Hence for an arbitrary $\eps>0$, $\lim_{{n}\to+\infty}\mb{P}(|\hat{\Psi}^{n}_0-\hat{\Psi}_0|>\eps)=0$ which means that for ${n}$ large enough, %(to be meant componentwise, as will be interpreted the expression ${n}\geq{m}$ for any other ${m}\in\mb{N}$),
\[\mb{P}(|\hat{\Psi}^{n}_0-\hat{\Psi}_0|>\eps)\leq\tfrac12.\] 
Since $\mc{F}_0$ consists of the $\mb{P}$-null sets and their complements only, we thus derive that there exists some ${m}_0\in\mb{N}$ such that $\mb{P}(|\hat{\Psi}^{n}_0-\hat{\Psi}_0|>\eps)=0$ for all ${n}\in\mb{N}$ with ${n}\geq{m}_0$. In particular, \[\lim_{{m}\to+\infty}\mb{P}\!\left(\left\{\sup_{{n}\geq{m}}|\hat{\Psi}^{n}_0-\hat{\Psi}_0|>\eps\right\}\right)=\lim_{\substack{{m}\to+\infty \\ {m}\geq {m}_0}}\mb{P}\!\left(\bigcup_{{n}\geq{m}}\big\{|\hat{\Psi}^{n}_0-\hat{\Psi}_0|>\eps\big\}\right)=0\] 
which is a well-known criterion for almost sure convergence. Hence $\lim_{{n}\to+\infty}\hat{\Psi}^{n}_0=\hat{\Psi}_0$ $\mb{P}$-a.s. which implies the convergence of $u^{n}(x)$ to $u(x)$. The convergence of $\left(u^{n}\right)'(x)$ to $u'(x)$ is then immediate.% since the value functions are concave.
\end{proof}

Similar arguments can be used to study the dual variables and we collect
the results together in the following theorem.

\begin{thm}
Suppose that Assumptions \ref{ass_contfilt}, \ref{ass_unif.expmom} and \ref{ass_prefconv} hold. Then
\begin{enumerate}
\item The processes
$(\hat{\kappa}^{n}-\hat{\kappa})\cdot M$, $n\in\mb{N}$, converge to zero in $\mc{M}^2$.
\item The processes $\hat{Y}^{n}\in\mc{Y}(y)$, $n\in\mb{N}$, converge to $\hat{Y}\equiv\hat{Y}^0\in\mc{Y}(y)$ in the semimartingale topology.
\item The functions $\ti{u}^{n}$, $n\in\mb{N}$, converge pointwise to $\ti{u}$. Their derivatives converge pointwise to $\ti{u}'$. 
\end{enumerate}
\end{thm}
\begin{proof}
Item (i) follows from the decomposition 
\[\hat{\kappa}^{n}=(1-p^{n})\hat{\nu}^{n}-\hat{Z}^{n}\] 
together with Theorems \ref{thm_stability} and \ref{thm_nuconv}. Item (i) and Theorem \ref{thm_stability} provides the convergence of 
$\Upsilon^n:=\hat{\kappa}^n\cdot M+\hat{N}^n$ to $\hat{\kappa}\cdot M+\hat{N}$ in $\mc{H}^\rho$ for all $\rho\geq1$. Convergence in the semimartingale topology then follows from Proposition \ref{semimarttop} (i) and (ii). 
For the last item observe that from Theorem \ref{thm_WZ09} we may write
\begin{equation*}
\ti{u}^{n}(y)=\ti{U}^{n}(y)\ti{c}^{n}_{p^{n}},\quad
\ti{c}^{n}_{p^{n}}=(c^{n}_{p^{n}})^{\frac{1}{1-p^{n}}}=e^{\,\frac{1}{1-p^{n}}\,\hat{\Psi}_0^{n}}\quad\mb{P}\text{-a.s.}
\end{equation*}
so that the claim is again a corollary of Theorem \ref{thm_stability}, as in the proof of Theorem \ref{thm_uconv}.
\end{proof}

\appendix
%%%%%%%%%%%%%%%%%%%%%%%%%%%%%%%%%%%%%%%%%%%%%%%%%%%%%%%%%%%%%%%%%%%%%
%Existence and uniqueness of the constraint utility maximization problem
%%%%%%%%%%%%%%%%%%%%%%%%%%%%%%%%%%%%%%%%%%%%%%%%%%%%%%%%%%%%%%%%%%%%%%

\section{Cone Constrained Utility Maximization}\label{existappend}
The utility maximization problem under polyhedral cone constraints is studied in detail in \cite{KZ03} and \cite{We09} in the additive framework. We hence work in this setting here and refer the reader to Remark \ref{addmultformat} for more details. We note that in the mentioned articles the constraint set $\mc{L}$ is independent of $(t,\omega)$. In this appendix we show how the results of \cite{We09} can be extended to give Theorem \ref{thm_WZ09}. The key result for the analysis above is that the dual optimizer is an element of our specific dual domain of supermartingale measures.

A careful reading of the proof of \cite{We09} Theorem 3.4.2 on existence and uniqueness shows that one needs one specific property of the cone $\U$ ($\mc{L}$ respectively), namely, provided that the set \[\mc{X}^{add}(1) \text{ is closed in the semimartingale topology}\] then the main existence result \cite{We09} Theorem 3.4.2 continues to hold with a predictably measurable, non-empty, closed convex multi-valued mapping $\U$ ($\mc{L}$ respectively).
\begin{lem}
Suppose that $\U$ satisfies Assumption \ref{asscones} then $\mc{X}^{add}(1)$ is closed in the semimartingale topology.
\end{lem}
\begin{proof}
Since $\U(t,\omega)$ (and hence $\mc{L}(t,\omega)$) is a polyhedral cone for all $t\in[0,T]$ $\mb{P}$-a.s. we see that \cite{CS09} Corollary 4.6 applies. This guarantees the result.
\end{proof}
We now adapt the results of \cite{We09} which are in the context of the utility maximization with a random endowment and begin with the primal
problem. 
\begin{lem}
Suppose that Assumptions \ref{ass_expmom} and \ref{asscones} hold. Then:
\begin{enumerate}
\item 
There exists an optimal terminal wealth $\hat{X}_T$, $\hat{X}\in\mc{X}^{add}(1),$ such that 
\begin{equation*}
u(1):=\sup_{X\in\mc{X}^{add}(1)}\,\E\!\left[U\big(X_T\big)\right]=
\E\big[U\big(\hat{X}_T\big)\big].
\end{equation*}
Moreover, any two such primal optimizers $\hat{X}$ and $\bar{X}$ are indistinguishable.
\item We have that $\hat{X}_T>0$ $\mb{P}$-a.s. so there is an optimal strategy $\hat{\nu}\in\mc{A}_\U$ with $\hat{X}=X^{1,\hat{\nu}}\in\mc{X}(1).$
\item The optimal strategy $\hat{\nu}$ is unique in the sense that given any other admissible strategy $\bar{\nu}$ with corresponding wealth process $X^{1,\bar{\nu}}_T$ which is optimal for the primal problem we have \[\E\big[\lo(\hat{\nu}-\bar{\nu})\cdot M,(\hat{\nu}-\bar{\nu})\cdot M\ro_T\big]=0.\]
\end{enumerate}
\end{lem}
\begin{proof}
From \cite{We09} Theorem 3.4.2 (iii) we see that there is an admissible $\hat{H}$ such that with $\hat{X}_T:=1+(\hat{H}\cdot S)_T$ 
\begin{equation*}
u(1)=\E\big[U\big(\hat{X}_T\big)\big].
\end{equation*}
%Observe that we may switch between the additive and exponential format since $\U$ is a cone and $S$ and \alert{$X\in\mc{X}(1)$ are assumed to be strictly positive, cf. the remarks \ref{addmultformat} and \ref{Xpositve}}.
Since $U$ is strictly concave a standard argument involving convex combinations gives the uniqueness at terminal time, cf. \cite{KS99} Lemma 3.3. For completeness we also derive the uniqueness on the level of processes. Let $\hat{X}$ and $\bar{X}$ be two primal optimizers, for which we know that $\hat{X}_T=\bar{X}_T$. Now suppose there is a $t\in[0,T)$ and a set $A\in\mc{F}_t$ such that $\hat{X}_t>\bar{X}_t$ on $A$ and $\mathbb{P}(A)> 0$. Define the integrand
\begin{equation*}
H:=\hat{H} \I{[0,t]}+\bar{H} \I{(t,T]} \I{A}+\hat{H}\I{(t,T]} \I{A^c},
\end{equation*}
where $\hat{H}$ and $\bar{H}$ are the integrands for $\hat{X}$ and $\bar{X}$. Observe that $X:=1+H\cdot S\in \mc{X}^{add}(1)$ as we have
%\begin{equation*}
$(H\cdot S)_u=(\bar{X}_u+\hat{X}_t-\bar{X}_t)\,\mathbf{1}_A+\hat{X}_u\,\mathbf{1}_{A^c}$
%\end{equation*}
for $u\geq t$ and this is nonnegative by assumption (recall that $\hat{X}_t>\bar{X}_t$ on $A$). Now we note that
$\hat{X}_T=\bar{X}_T$ $\mb{P}$-a.s. and write
\begin{align*}
\E[U(X_T)]%=\E\!\left[\E\!\left[U(X_T)\big|\,\mc{F}_t\right]\right]
=\E\!\left[\I{A^c}\E\!\left[U\big(\hat{X}_T\big)\Big|\,\mc{F}_t\right]
+\I{A}\E\!\left[U\big(\bar{X}_T+\hat{X}_t-\bar{X}_t\big)\Big|\,\mc{F}_t\right]\right]%\\
%&>\E\!\left[\I{A^c}\E\left[U\big(\bar{X}_T\big)\Big|\,\mc{F}_t\right]
%+\I{A}\E\!\left[U\big(\bar{X}_T\big)\Big|\,\mc{F}_t\right]\right]=
>\E\!\left[U\big(\bar{X}_T\big)\right]=u(1).
\end{align*}
This is a contradiction and the result in (i) follows from the continuity of the wealth processes. 

For item (ii) observe from \cite{We09} Theorem 3.4.2 (iv) that $\hat{X}_T=-\ti{U}'\!\left(\tfrac{d\hat{\zeta}_c}{d\mb{P}}\right)$ where $\hat{\zeta}_c$ is a finite, nonnegative and countably additive measure that is absolutely continuous with respect to $\mb{P}$. Since $-\ti{U}'(y)=0$ if and only if $y=+\infty$ for $y\geq0$ we cannot have that $\hat{X}_T$ is zero on a set of nonzero $\mb{P}$-measure, this would contradict the finiteness of $\hat{\zeta}_c$. 

For item (iii) we have the equality
%\begin{equation*}
$\mc{E}\!\left(\hat{\nu}\cdot M +\hat{\nu}\cdot \lo M,M \ro\lambda\right) \equiv
\mc{E}\!\left(\bar{\nu}\cdot M +\bar{\nu}\cdot \lo M,M \ro\lambda\right)$.
%\end {equation*}
By the uniqueness of the stochastic logarithm we derive that
%\begin{equation*}
$\hat{\nu}\cdot M +\hat{\nu}\cdot \lo M,M \ro\lambda \equiv \bar{\nu}\cdot M
+\bar{\nu}\cdot \lo M,M \ro\lambda$
%\end {equation*}
and thus it follows that $(\hat{\nu}-\bar{\nu})\cdot M$ is a continuous local
martingale of finite variation and is hence constant and equal to
zero, which proves the last assertion.
\end{proof}

In \cite{We09} the dual domain is a subset of $L^\infty(\mb{P})^*$, the bounded, finitely additive measures that are absolutely continuous with respect to $\mb{P}$, which contains $\mc{Y}^{add}(y)$ where
\begin{equation*}
\mc{Y}^{add}(y):=\left\{Y\geq 0\,|\,Y_0=y \text{ and } XY \text{
is a supermartingale for all }X\in\mc{X}^{add}(1)\right\}\!.
\end{equation*}
Note that $\mc{Y}^{add}(y)$ depends on $\mc{L}$ (respectively $\mc{K}$). The next lemma, which shows that the dual minimizer of \cite{We09} can be related to an element of $\mc{Y}^{add}(y)$, is key.

\begin{lem}
Let the assumptions of the previous lemma hold. Then, given $y>0$, there is a $\hat{Y}^{y}\in\mc{Y}^{add}(y)$ which is optimal for the dual problem, unique up to indistinguishability and satisfying $\hat{Y}^{y}_T>0$, $\mb{P}$-a.s.
\end{lem}
\begin{proof}
Define the sets
\begin{gather*}
\mc{C}:=\{\xi\in L^0(\mb{P})\,|\,0\leq \xi\leq X_T, \, X\in\mc{X}^{add}(1)\}\\
\mc{D}:=\{\eta\in L^0(\mb{P})\,|\,0\leq \eta\leq Y_T, \, Y\in\mc{Y}^{add}(1)\}.
\end{gather*}
By construction and the above lemma we have 
\begin{equation*}
u(1)=\E\big[U\big(\hat{X}_T\big)\big]=\sup_{\xi\in\mc{C}}\,\E\big[U(\xi)\big]
\end{equation*}
and thus, using the Calculus of Variations argument from the proof of Bouchard and Pham \cite{BP04} Lemma 5.7, one can show that with $\ti{\eta}:=U'(\hat{X}_T)>0$ we have 
%\begin{equation*}
$\E\big[\ti{\eta}\,(\hat{X}_T-\xi)\big]\geq 0\text{ for all }\xi\in\mc{C}$.
%\end{equation*}
We set $y:=\E\big[\ti{\eta}\hat{X}_T\big]=\E\big[(\hat{X}_T)^p\big]>0$ and observe that 
%\begin{equation*}
$\E[\ti{\eta}\,\xi]\leq y\text{ for all }\xi\in\mc{C}$.
%\end{equation*}
Hence $\ti{\eta}/y\in\mc{C}^\circ$, where we write $\mc{C}^\circ$ for the polar  of the cone $\mc{C}$, 
\begin{equation*}
\mc{C}^\circ:=\{\eta\in L^0_+(\mb{P})\,\big|\,\E\!\left[\xi\eta\right]\leq 1\text{ for all }\xi\in\mc{C}\}.
\end{equation*}
Observing from \cite{We09} Lemma 3.5.7 that $\xi\in\mc{C}$ if and only if %\begin{equation*}
$\xi\geq0 \text{ and } \E_{\mb{Q}}[\xi]\leq 1 \text{ for all }\mb{Q}\in\mc{M}^{\text{sup}}$,
%\end{equation*}
we derive that $\mc{C}=(\mc{M}^{\text{sup}})^\circ$, where 
\begin{equation*} 
\mc{M}^{\text{sup}}:=\{\mb{Q}\sim \mb{P}\,|\, X\text{ is a $\mb{Q}$-supermartingale for all }X\in\mc{X}^{add}(1)\}.
\end{equation*}
%we see that $\mc{C}$ is closed in measure.
Applying the same reasoning as in the proof of \cite{KS99} Lemma 4.1 we derive that $\mc{D}^{\circ\circ}=\mc{D}$ and equating measures $\mb{Q}$ with their densities $Z^\mb{Q}$ we are led to conclude that $\mc{M}^\text{sup}\subset \mc{D}$. Hence $\mc{C}^\circ=(\mc{M}^{\text{sup}})^{\circ\circ}\subset\mc{D}^{\circ\circ}=\mc{D}$ from which $\ti{\eta}/y\in\mc{D}$.
%$(\mc{M}^\text{sup})^{\circ\circ}\subset \mc{D}^{\circ\circ}=\mc{D}$. Next $\mc{C}\subset \mc{D}^\circ$ implies $\mc{D}^{\circ\circ}\subset\mc{C}^\circ$, so $\mc{D}\subset\mc{C}^\circ$, from which we conclude that $\mc{C}^\circ=\mc{D}$ and $\frac{\hat{Y}_T}{y}\in \mc{D}$. 
%$\mc{C}^\circ=(\mc{M}^\text{sup})^{\circ\circ}\subset\mc{D}^{\circ\circ}=\mc{D}$ from which $\frac{\hat{Y}_T}{y}\in \mc{D}$. The optimality of $\hat{Y}_T$ now follows from \cite{HMUW10} Theorem 2.6.
Thus there is a $\hat{Y}\in\mc{Y}^{add}(1)$ with $0<\ti{\eta}/y\leq\hat{Y}_T$ and such that \[1=\E\big[\hat{X}_0\hat{Y}_0\big]\geq\E\big[\hat{X}_T\hat{Y}_T\big]\geq\E\big[\hat{X}_T\ti{\eta}/y\big]=1.\]
In particular $\hat{X}\hat{Y}$ is a martingale. We conclude that $\hat{Y}^y:=y\hat{Y}\in\mc{Y}^{add}(y)$ is a dual optimizer. More explicitly, since $\ti{\eta}=U'(\hat{X}_T)$,
\begin{align*}
\E\big[\ti{U}(\hat{Y}^y_T)\big]&\geq\inf_{Y\in\mc{Y}^{add}(y)}\E\big[\ti{U}(Y)\big]\geq \inf_{Y\in\mc{Y}^{add}(y)}\E\big[U(\hat{X}_T)-\hat{X}_TY\big]\geq\E\big[U(\hat{X}_T)\big]-y\\
&=\E\big[U(\hat{X}_T)\big]-\E\big[\hat{X}_T\ti{\eta}\big]=\E\big[\ti{U}(\ti{\eta})\big]\geq\E\big[\ti{U}(\hat{Y}^y_T)\big].
\end{align*}
For uniqueness we again suppose that there exists a $t\in[0,T)$ and a set $A\in\mc{F}_t$ such that $\hat{Y}_t>\bar{Y}_t$ on $A$ and $\mathbb{P}(A)>0$, where $\hat{Y}$ and $\bar{Y}$ are two optimal dual processes that are necessarily equal at terminal time $T$.
Since the dual function is strictly decreasing we have that the following inequality holds on $A$,
\begin{equation*}
\E\Big[\ti{U}\!\left(\tfrac{\hat{Y}_t}{\bar{Y}_t}\,\bar{Y}_T\right)\Big|\,\mc{F}_t\Big]<
\E\!\left[\ti{U}\big(\bar{Y}_T\big)\Big|\,\mc{F}_t\right]=\E\Big[\ti{U}(\hat{Y}_T)\Big|\,\mc{F}_t\Big].
\end{equation*}
Note that $\bar{Y}$ being a supermartingale $\bar{Y}_T>0$ implies that $\bar{Y}>0$.
We then define the process
\begin{equation*}
Y:=\hat{Y}\,\I{[0,t]}+\tfrac{\hat{Y}_t}{\bar{Y}_t}\,\bar{Y}\,\I{A}\,\I{(t,T]}+\hat{Y}\,\I{A^c}\,\I{(t,T]}.
\end{equation*}
It is now essential to show that $Y\in\mc{Y}^{add}(1)$ which holds by separately checking the respective cases due to the choice of $\mc{Y}^{add}(1)$ as a family of supermartingale measures for $S$, more precisely, %$t<r$, $r\leq t<s$, and $s\leq t$ one can verify that for all $s,r\in[0,T]$ with $r<s$ we have
%\begin{equation*}
%\E[X_sY_s\,|\,\mc{F}_r]\leq X_rY_r,
%\end{equation*}
$XY$ is a supermartingale for any admissible wealth process $X$. Note that here it is also important that the $Y$ constructed above is right-continuous at $t$. A similar computation as for the uniqueness of $\hat{X}$ then %This now leads to the following computation,
%\begin{align*}
%\E\Big[\ti{U}(Y_T)\Big]%&=\E\!\left[\E\!\left[\ti{U}(Y_T)\Big|\,\mc{F}_t\right]\right]=\E\!\left[\I{A^c}\E\!\left[\ti{U}(\hat{Y}_T)\Big|\,\mc{F}_t\right]
%%+\I{A}\E\!\left[\ti{U}\!\left(\frac{\bar{Y}_T}{\bar{Y}_t}\,\hat{Y}_t\right)\Bigg|\,\mc{F}_t\right]\right]\\
%<\E\!\left[\I{A^c}\E\!\left[\ti{U}(\hat{Y}_T)\Big|\,\mc{F}_t\right]
%+\I{A}\E\!\left[\ti{U}(\hat{Y}_T)\Big|\,\mc{F}_t\right]\right]=\E\!\left[\ti{U}(\hat{Y}_T)\right]=\ti{u}(y),
%\end{align*}
%which 
results in a contradiction. The processes $\hat{Y}$ and $\bar{Y}$ are c{\`a}dl{\`a}g %(they are supermartingales) 
and they satisfy $\hat{Y}_t=\bar{Y}_t$ a.s. for each $t\in[0,T]$. We then conclude that they are indistinguishable.% and this completes the proof.
\end{proof}
The remaining items from Theorem \ref{thm_WZ09}, if not already implicitly contained in the previous proofs, can be deduced in a standard fashion so we omit the details.\\ \\
In the above study we relate the optimal processes from the utility maximization problem to solutions of quadratic BSDEs and rely on  
%This step is crucial, it enables us to carry over stability results for BSDEs to stability assertions for the optimal wealth processes, the optimal strategies and the value functions. 
the fact that $\hat{\Psi}\in\mf{E}$. For the proof of the latter it turns out to be convenient to use a characterization of $\hat{\Psi}$, given by the so-called \emph{primal and dual opportunity process}. To this end first define the domain
\[\mc{Y}^*(y):=\mc{Y}^{add}(y)\cap\{Y>0\},\]
which in view of $\hat{Y}^1\in\mc{Y}^*(1)$ does not affect the optimizers. In addition define, for $t\in[0,T]$, the continuation strategies
\begin{equation*}
\mc{Y}^*(Y,t) = \{\ti{Y}\in\mc{Y}^*(y):\ti{Y}=Y \text{ on } [0,t]\}.
\end{equation*}
We then can state
\begin{prop}[Nutz \cite{Nu109} Proposition 3.1]
There is a unique c{\`a}dl{\`a}g semimartingale $L^{\textit{op}}$, the opportunity process, such that for any admissible strategy $\nu\in\mc{A}_\U$ and $t\in[0,T]$
\begin{equation}\label{UtDynLt}
L^{\textit{op}}_t\,U(X_t^\nu)=\esssup_{\substack{\check{\nu}\in\,\mc{A}_{\U,\nu}}}\E\!\left[U\!\Big(X_T^{\check{\nu}}\Big)\Big|\,\mc{F}_t\right],
\end{equation}
where the optimization is over all the continuation strategies $\check{\nu}\in\,\mc{A}_{\U,\nu}$ for $\nu$, i.e. over all the admissible strategies $\check{\nu}$ that are equal to $\nu$ on $[0,t]$. If $(\hat{X},\hat{Y})$ denotes the optimal pair for the utility maximization problem satisfying $\hat{Y}_0=u'(\hat{X}_0)$ then $\hat{Y}=L^{\textit{op}}\,U'(\hat{X})$. In particular, $\hat{\Psi}=\log(L^{\textit{op}}).$
\end{prop}
By our \emph{specific} choice of the dual domain, mimicking the proof of \cite{Nu109} Proposition 4.3, one can also show the following result.
\begin{prop}
There exists a unique c{\`a}dl{\`a}g process $\ti{L}^\textit{op}$, the dual opportunity process, such that for any $Y\in\mc{Y}^*(y)$ and $t\in[0,T]$
\begin{equation}\label{DualOpp}
\ti{L}^{\textit{op}}_t\,\ti{U}(Y_t)=\essinf_{\substack{\check{Y}\in\mc{Y}^*(Y,t)}}\E\!\left[\ti{U}\!\Big(\check{Y}_T\Big)\Big|\,\mc{F}_t\right],
\end{equation}
%where the optimization is over all the continuations $\check{Y}\in\mc{Y}^*(y)$ of $Y$ onto $(t,T]$. 
Moreover, the minimum is attained at $Y=\hat{Y}$ and we have that $\ti{L}^{\textit{op}}=(L^{\textit{op}})^{\frac{1}{1-p}}$.
\end{prop}
The previous two propositions allow us to prove the required estimates on $\hat{\Psi}=\log\Bigl(\frac{u'(x)\hat{Y}^1}{U'(\hat{X})}\Bigr)$.
\begin{lem}\label{PsiInE}
Let Assumptions \ref{ass_expmom} and \ref{asscones} hold, then $\hat{\Psi}\in\mf{E}$.
\end{lem}
\begin{proof}
Let $p\in(0,1)$ so that $q=\frac{p}{p-1}\in(-\infty,0)$ and $L^\textit{op}\geq 1$. The last inequality follows from \eqref{UtDynLt} by using the strategy $\nu\equiv0$. %%, namely if $x>0$
%\begin{equation}\label{Lpxp}
%L_t\,\frac{1}{p}\,x^p=\esssup_{\substack{\check{\nu}\in\,\mc{A}_{\U,\nu}}}\E\!\left[U\!\Big(X_T^{\check{\nu}}\Big)\Big|\,\mc{F}_t\right]\geq\E\!\left[U\!\Big(X_T^{\check{\nu}\equiv0}\Big)\Big|\,\mc{F}_t\right]=\frac{1}{p}\,x^p.
%\end{equation}
In particular $\hat{\Psi}\geq0$ and we notice that for all
$\delta>0$
\begin{equation}\label{expPsiLdelta}
\E\!\left[\exp\!\left(\delta\hat{\Psi}^*\right)\!\right]=\E\Bigg[\sup_{\substack{0\leq t\leq T}}\left(\exp\!\left(\delta\hat{\Psi}_t\right)\!\right)\Bigg]=\E\!\left[\Big(\big(L^\textit{op}\big)^\delta\Big)^*\right].
\end{equation}
In what follows the constant $c_{p,\delta}>0$ is generic, depends on $p$ and $\delta$ and may change from line to line. Let us consider an exponential moment of $\lo\lambda\cdot M,\lambda\cdot M\ro_T$ of order $k>k_q:=q^2-\frac{q}{2}-q\sqrt{q^2-q}$. We now set $\beta:=1-\frac{1}{q}\sqrt{q^2-q}>1$, $\varrho:=\beta/(\beta-1)>1$ and $\delta:=k\varrho/k_q>1$. After defining $Y^\lambda:=\mc{E}(-\lambda\cdot M)$ %, which is seen to be in
%$\mc{Y}(1)$ by a simple application of the product rule
we deduce from \eqref{UtDynLt} that for a fixed strategy $\nu\in\mc{A}_\U$, denoting by $\check{\nu}$ a
time-$t$ continuation strategy of $\nu$,
\begin{align*}
\big(L_t^\textit{op}\big)^\delta&%=\left(\esssup_{\substack{\check{\nu}\in\,\mc{A}_\nu}}\E\!\left[U\!\Big(X_T^{\check{\nu}}\Big)\!\Big/U\!\Big(X_t^{\nu}\Big)\Big|\,\mc{F}_t\right]\right)^\delta=p^\delta\esssup_{\substack{\check{\nu}\in\,\mc{A}_\nu}}\E\!\left[U\!\Big(X_T^{\check{\nu}}\!\big/X_t^{\check{\nu}}\Big)\Big|\,\mc{F}_t\right]^\delta\\&
\leq
p^\delta\esssup_{\substack{\check{\nu}\in\,\mc{A}_{\U,\nu}}}\!\bigg(\E\!\left[\ti{U}\!\Big(Y_T^{\lambda}\!\big/Y_t^{\lambda}\Big)\Big|\,\mc{F}_t\right]+
\E\!\left[\Big(X_T^{\check{\nu}}\!\big/X_t^{\check{\nu}}\Big)\Big(Y_T^{\lambda}\!\big/Y_t^{\lambda}\Big)\Big|\,\mc{F}_t\right]
\bigg)^\delta\\&\leq
%p^\delta\esssup_{\substack{\check{\nu}\in\,\mc{A}_\nu}}\!\bigg(\E\!\left[\ti{U}\!\Big(Y_T^{\lambda}\!\big/Y_t^{\lambda}\Big)\Big|\,\mc{F}_t\right]+
%1 \bigg)^\delta\leq
%c_{p,\delta}\bigg(\Big(-\tfrac{1}{q}\Big)^\delta\E\!\left[\Big(Y_T^{\lambda}\!\big/Y_t^{\lambda}\Big)^q\bigg|\,\mc{F}_t\right]^\delta+
%1\bigg)\\
%&\leq
 c_{p,\delta}\,\E\bigg[\mc{E}(-\beta q\lambda\cdot
M)_{t,T}^{1/\beta}\,
\exp\!\Big(k_q\lo\lambda\cdot M,\lambda\cdot M\ro_{t,T}\Big)^{1/\varrho}\bigg|\,\mc{F}_t\bigg]^\delta\!+ c_{p,\delta}\\
&\leq c_{p,\delta}\,
\E\bigg[\exp\!\Big(k_q\lo\lambda\cdot M,\lambda\cdot M\ro_T\Big)\,\bigg|\,\mc{F}_t\bigg]^{\delta/\varrho}\!+ c_{p,\delta}
=:c_{p,\delta}
\left(\chi_t^{\delta/\varrho}+1\right),
\end{align*}
by making use of the definition of $\ti{U}$, the supermartingale property of $Y^\lambda X^{\check{\nu}}$ and $\mc{E}(-\beta q\lambda\cdot M)$,
H\"{o}lder's inequality and the positiveness of $-{1}/{q}$ and $k_q$. % It follows by
Thanks to the assumption on the exponential moment of $\lo\lambda\cdot M,\lambda\cdot M\ro_T$, the process $\chi$ is a (nonnegative) martingale on $[0,T]$ and thus amenable to Doob's inequality from which the result follows.%. We deduce that
%If $\delta\in(0,2]$, then we find the result verified upon using Jensen's inequality and the above estimate for $\delta=3$.
%\begin{gather*}
%\E\Bigg[\sup_{\substack{0\leq t\leq T}}L_t^\delta\Bigg]^{3/\delta}\leq \E\Bigg[\sup_{\substack{0\leq t\leq T}}L_t^3\Bigg]<+\infty,
%\end{gather*}
%and we find the result verified upon recalling \eqref{expPsiLdelta}.\\

Let us now turn to the case of $p<0$, i.e. when $q=\frac{p}{p-1}\in(0,1)$ and $0<L^\textit{op}\leq 1$. Take an exponential moment of $\lo \lambda\cdot M,\lambda\cdot M\ro_T$ of order $k>(1-p)k_q>k_q:=q^2+\frac{q}{2}+\sqrt{q^2+q}$. We define $\delta:=\frac{k\varrho}{(1-p)k_q}>1$ where $\beta:=1+\frac{1}{q}\sqrt{q^2+q}>1$ and $\varrho:=\beta/(\beta-1)>1$. Then% we have using again \eqref{Lpxp}. As above we consider $\delta>0$ and derive
\begin{align*}
\E\!\left[\exp\!\left(\delta\hat{\Psi}^*\right)\!\right]&=
\E\Bigg[\!\!\left(\exp\!\left(\delta\sup_{\substack{0\leq t\leq T}}\left(-\hat{\Psi}_t\right)\right)\!\right)\Bigg]=
\E\left[\left(\big(\ti{L}^\textit{op}\big)^{-\delta(1-p)}\right)^*\right]\\&\leq\E\bigg[\exp\!\Big(k_q\lo\lambda\cdot M,\lambda\cdot M\ro_T\Big)\,\bigg|\,\mc{F}_t\bigg]^{\delta(1-p)/\varrho}\!\!,
\end{align*}
where $\ti{L}^\textit{op}$ is the dual opportunity process. The claim can then again be deduced from Doob's inequality.
\end{proof}

%%%%%%%%%%%%%%%%%%%%%%%%%%%%%%%%%%%%%%%%%%%%%%%%%%%%%%%%%%%%%%%%%%%%%
%Appendix B: BSDE Results
%%%%%%%%%%%%%%%%%%%%%%%%%%%%%%%%%%%%%%%%%%%%%%%%%%%%%%%%%%%%%%%%%%%%%%

\section{Semimartingale BSDEs Under Exponential Moments}\label{appendBSDE}
In this appendix we summarize the existence, uniqueness and stability results for quadratic semimartingale BSDEs under exponential moments as described in \cite{MW10} to which we refer for proofs. Note that these results generalize the corresponding results provided in Briand and Hu \cite{BH08} for the Brownian framework. Let us consider the BSDE on $[0,T]$,
\begin{equation}
d\Psi_t= Z_t^\tr\,dM_t + dN_t -F(t,Z_t)\,dA_t-\frac{1}{2}\,d\lo N,N\ro_t ,\quad \Psi_T=\xi,\label{BSDEF}
\end{equation}
where $F:[0,T]\times\Omega\times\mb{R}^d\to\mb{R}$ is a random predictable function and $\xi$ is an $\mc{F}_T$-measurable random variable.
A \emph{solution} to the BSDE \eqref{BSDEF} is defined as in Definition \ref{DefnBSDESol}. We require the following assumption.
\begin{ass}\label{assBSDE}
There exist positive numbers $\gamma\geq 1$ and $\delta$ together with an $M$-integrable $\mb{R}^d$-valued process $\ti{\lambda}$ so that for
\begin{equation*}
\alpha:=\|B\ti{\lambda}\|^2 \text{ and } \,|\alpha|_1:=\int_0^T\alpha_t\,dA_t=\int_0^T\ti{\lambda}^\tr_t\,d\lo M,M\ro_t\ti{\lambda}_t
\end{equation*}
we have (a.s. when appropriate)
\begin{enumerate}
\item The random variable $|\xi|+|\alpha|_1$ has exponential moments of all orders.%, i.e.
%for all $c>0$
%\begin{equation}
%\E\!\left[\exp\!\Big(p\,\big[|\xi|+|\alpha|_1\big]\Big)\right]<+\infty.\label{AssExpMom}
%\end{equation}
\item For all $t\in[0,T]$ the driver $z\mapsto F(t,z)$ is continuous and convex in $z$.
\item The generator $F$ satisfies a quadratic growth condition in $z$, i.e. for all $t$ and $z$ we have
\begin{equation}
|F(t,z)|\leq \alpha_t+\frac{\gamma}{2}\|B_tz\|^2.\label{AssGrowF}
\end{equation}
\item The function $F$ is locally Lipschitz in $z$, i.e. for all $t, z_1$ and $z_2$
\begin{equation*}
|F(t,z_1)-F(t,z_2)|\leq \delta \Big(\|B_t\ti{\lambda}_t\| +\|B_tz_1\|+\|B_tz_2\|\Big)\|B_t(z_1-z_2)\|.\label{AssLocLipF}
\end{equation*}
\end{enumerate}
If this assumption is satisfied we refer to \eqref{BSDEF} as BSDE$(F,\xi)$ with the set of parameters $(\alpha,\gamma,\delta)$.
\end{ass}
The following two results collect the key results used in the present article. 
\begin{thm}
\label{thmMW10} Suppose Assumption \ref{assBSDE} holds.
\begin{enumerate}
\item If $(\Psi,Z,N)$ solves the BSDE \eqref{BSDEF} with $\Psi\in\mathfrak{E}$ then $Z\cdot M$ and $N$ are in $\mc{M}^\rho$ for $\rho\geq1$.
\item If $(\Psi,Z,N)$ and $(\Psi',Z',N')$ are both solutions to the BSDE \eqref{BSDEF} with $\Psi,\Psi'\in\mathfrak{E}$ then $\Psi$ and $\Psi'$, $Z\cdot M$ and $Z'\cdot M$ as well as $N$ and $N'$ are indistinguishable.
\end{enumerate}
\end{thm}

\begin{thm}[Stability]\label{ThmStab}
Consider a family of BSDEs($F^n,\xi^n$) for $n\in\mb{N}_0$ for which Assumption
\ref{assBSDE} holds with parameters $(\alpha^n,\gamma,\delta)$. Assume that the
exponential moments assumption holds uniformly
in $n$, i.e. for all $c>0$,
\begin{equation*}
\sup_{\substack{n\geq 0}}\E\!\left[e^{c\,(|\xi^n|+|\alpha^n|_1)}\right]<+\infty.
\end{equation*}
If for $n\geq0$ $(\Psi^n,Z^n,N^n)$ is the solution to the BSDE($F^n,\xi^n$)  with $\Psi\in\mathfrak{E}$ and if
\begin{equation}
|\xi^n-\xi^0| + \int_0^T\big|F^n-F^0\big|\,(s,\Psi_s^0,Z_s^0)\,dA_s\longrightarrow 0 \quad\text{ in
probability, as }n\to+\infty,
\end{equation}
then for each $\rho\geq1$ as $n\to+\infty$
\begin{gather*}
%\lim_{{n}\to+\infty}\E\!\left[\exp\!\left(\rho\sup_{0\leq t\leq T}|\hat{\Psi}^{n}_t-\hat{\Psi}_t|\right)\right]=1,\\
\lim_{{n}\to+\infty}\E\!\left[\exp\!\left(\rho\left(\Psi^{n}-\Psi^0\right)^*\right)\right]=1,\\
\lim_{{n}\to+\infty}\E\!\left[\left( \big\lo(Z^{n}-Z^0)\cdot M,(Z^{n}-Z^0)\cdot M\big\ro_T
+\lo N^{n}-N^0,N^{n}-N^0\ro_T\right)^{\rho/2}\right]=0,
\end{gather*}
%\begin{gather*}
%\E\Bigg[\!\!\left(\exp\!\left(\sup_{\substack{0\leq t\leq T}}\big|\Psi_t^n-\Psi_t^0\big|\right)\!\right)^\rho\Bigg]\longrightarrow 1\quad\text{and}\quad Z^n\cdot M+N^n\longrightarrow Z^0\cdot M+N^0 \text{ in } \mc{M}^\rho.
%\end{gather*}
\end{thm}

%%%%%%%%%%%%%%%%%%%%%%%%%%%%%%%%%%%%%%%%%%%%%%%%%%%%%%%%%%%%%%%%
%Appendix C: Continuity with respect to the Constraints
%%%%%%%%%%%%%%%%%%%%%%%%%%%%%%%%%%%%%%%%%%%%%%%%%%%%%%%%%%%

\section{Set Valued Analysis}\label{appendKur}
In this appendix we provide the necessary definitions from set valued analysis relevant to the present article. We fix a sequence $(\mc{J}^n)_{n\in\mb{N}}$ of closed and convex subsets of $\mathbb{R}^d$ and begin with the analogue of $\liminf$ and $\limsup$ for sets, cf. Aubin and Frankowska \cite{AF90}.

\begin{defn}
The \emph{upper limit} of the sequence $(\mc{J}^n)_{n\in\mb{N}}$ is the subset
\begin{align*}
\Limsup_{n\to+\infty}\,\mc{J}^n:=&\left\{ x\in\mb{R}^d\,\bigg|\,\liminf_{n\to+\infty}\mathrm{dist}(x,\mc{J}^n)=0\right\}\\
 = &\left\{ x\in\mb{R}^d\,\bigg|\,x \text{ a cluster point of an }(x_n)_{n\in\mb{N}},\text{ } x_n\in\mc{J}^n\text{ for all } n\in \mb{N}\right\}\!,
\end{align*}
where dist denotes the usual distance function from a set in $\mb{R}^d$. Similarly, the \emph{lower limit} of the sequence $(\mc{J}^n)_{n\in\mb{N}}$ is the subset
\begin{align*}
\Liminf_{n\to+\infty}\mc{J}^n:=&
\left\{ x\in\mb{R}^d\,\bigg|\,\lim_{n\to+\infty}\mathrm{dist}(x,\mc{J}^n)=0\right\}\\
=&\left\{ x\in\mb{R}^d\,\bigg|\,x =\lim_{n\to+ \infty}x_n,\text{ where } x_n\in\mc{J}^n\text{ for all } n\in \mb{N}\right\}.
\end{align*}
A set $\mc{J}$ is called the set \emph{limit} of the sequence $(\mc{J}^n)_{n\in\mb{N}}$ if the upper and lower limit sets coincide,
i.e.   
\begin{equation*}  
\mc{J} = \Limsup_{n\to+\infty}\mc{J}^n = \Liminf_{n\to+\infty}\mc{J}^n,
\end{equation*}
in which case we write $\mc{J}= \Lim_{n\to+\infty} \mc{J}^n$.
\end{defn}
We note that if $(\mc{J}^n)_{n\in\mb{N}}$ is a sequence of closed convex predictably measurable multivalued mappings then both $\Limsup_{n\to+\infty}\mc{J}^n$ and $\Liminf_{n\to+\infty}\mc{J}^n$ are convex predictably measurable multivalued mappings.

The following proposition shows that the above notion of set convergence implies pointwise convergence of the associated projections. In fact, according to Schochetman and Smith \cite{SS92} Theorem 3.3, everywhere pointwise convergence of the nearest point operators is \emph{equivalent} to the above set convergence, which is often called Kuratowski convergence in the literature. This motivates the choice of the Kuratowski convergence as an appropriate notion of convergence of sets. 
\begin{prop}[Schochetman and Smith \cite{SS92} Theorem 3.2]\label{kuraconv}
If $\Pi$ denotes the nearest point operator onto the indicated (closed and convex) set, then if the sequence $(\mc{J}^n)_{n\in\mb{N}}$ has a set limit denoted by $\mc{J}$ then the sequence $(\Pi_{\mc{J}^n})_{n\in\mb{N}}$ of mappings converges pointwise on $\mathbb{R}^d$ to $\Pi_{\mc{J}}$. 
\end{prop}
The final proposition shows that the alternative assumption given in Remark \ref{rmk_KK}
and used in \cite{K10} also leads to the appropriate convergence of the projections.
\begin{prop}
\label{prop_Kard}
Let the sequence $(\mc{J}^n)_{n\in\mb{N}}$ have a set limit denoted by $\mc{J}$ and suppose that $Q$ is a $d\times d$ matrix
such that $\ker(Q)\subseteq\mc{J}^n$ for all $n\in\mb{N}$ and $\ker(Q)\subseteq\mc{J}$. Then $Q\mc{J} = \Lim_{n\to+\infty} Q\mc{J}^n$.
\end{prop} 
\begin{proof}
%Again by fixing an appropriate null set we can assume that $Q,\mc{J}^n$ and $\mc{J}$ are constant and thus omit the arguments $(t,\omega)$.
We must show that 
\begin{equation*}
Q\mc{J}\subseteq \Liminf_{n\to+\infty}Q\mc{J}^n\subseteq  \Limsup_{n\to+\infty}Q\mc{J}^n\subseteq Q\mc{J}.
\end{equation*}
The first containment is an easy consequence of the definitions and we omit the details. Since one always has 
$\Liminf_{n\to+\infty}Q\mc{J}^n\subseteq  \Limsup_{n\to+\infty}Q\mc{J}^n$ we need only prove the final containment.

Let $y\in \Limsup_{n\to+\infty}Q\mc{J}^n$, this means we may find sequences $(y_n)_{n\in\mb{N}}$ and $(x_n)_{n\in\mb{N}}$ for which $y_n=Qx_n$ and $x_n\in\mc{J}^n$ for all $n\in\mb{N}$ and such that $(y_{n_k})_{k\in\mb{N}}$ converges to $y$ for a subsequence $(n_k)_{k\in\mb{N}}$. We must show that we can construct $x$ with $x\in\mc{J}$ and 
$Qx=y$. For each $k\in\mb{N}$ we may decompose $x_{n_k}$ uniquely as $x_{n_k}=x_{n_k}^1+x_{n_k}^2$ with $x_{n_k}^1\in\ker(Q)$ and $x_{n_k}^2\in\ker(Q)^{\perp}$.
From the assumption $\ker(Q)\subset \mc{J}^n$ we see that for all $\eps\in(0,1)$, $\tfrac{-(1-\eps)}{\eps}\,x_{n_k}^1\in\mc{J}^{n_k}$ so that 
\begin{equation*}
(1-\eps)x_{n_k}^2=\eps\,\tfrac{-(1-\eps)}{\eps}\,x_{n_k}^1+(1-\eps)x_{n_k}\in\mc{J}^{n_k}
\end{equation*}
by convexity. Since each $\mc{J}^{n_k}$ is also closed, letting $\eps$ tend to zero we see $x_{n_k}^2\in\mc{J}^{n_k}$.
From the above construction it follows that $x_{n_k}^2=Q^{\dagger}Qx_{n_k}$, where $Q^{\dagger}$
is the Moore-Penrose pseudoinverse of $Q$. Define now the vector $x:=Q^{\dagger}y$, then we have $x=\lim_{k\to+\infty}x_{n_k}^2$ since
\begin{equation*}
\big\|x_{n_k}^2-x\big\|\leq\big\|Q^{\dagger}\big\|\cdot\big\|Qx_{n_k}^2-y\big\|=\big\|Q^{\dagger}\big\|\cdot\|Qx_{n_k}-y\|=\big\|Q^{\dagger}\big\|\cdot\|y_{n_k}-y\|,
\end{equation*} 
where the right hand side tends to zero by assumption. As a consequence $x\in\Limsup_{n\to+\infty}\mc{J}^n=\mc{J}$ and $y=\lim_{k\to+\infty}Qx_{n_k}=\lim_{k\to+\infty}Qx_{n_k}^2=Qx$, hence $y\in Q\mc{J}$.
\end{proof}

\begin{ak}
The authors thank Ulrich Horst and Harry Zheng for helpful suggestions, comments and discussion.
\end{ak}
\bibliography{Sensitivity_Cones}
\bibliographystyle{abbrv}

\end{document}